\documentclass{amsart}
\usepackage{amssymb}
\usepackage{mathrsfs}

\newtheorem{theorem}{Theorem}[section]
\newtheorem{lemma}[theorem]{Lemma}
\newtheorem{proposition}[theorem]{Proposition}
\newtheorem{corollary}[theorem]{Corollary}
\newtheorem{claim}[theorem]{Claim}

\theoremstyle{definition}
\newtheorem{definition}[theorem]{Definition}

\theoremstyle{remark}
\newtheorem{remark}[theorem]{Remark}
\newtheorem{convention}[theorem]{Convention}

\numberwithin{equation}{section}

\DeclareMathOperator{\Int}{Int}
\DeclareMathOperator{\pr}{pr}
\newcommand{\Ar}{\mathscr{A}}
\newcommand{\D}{\mathscr{D}}
\newcommand{\F}{\mathcal{F}}
\newcommand{\I}{\mathcal{I}}
\newcommand{\K}{\mathcal{K}}
\newcommand{\N}{\mathbb{N}}
\newcommand{\R}{\mathbb{R}}
\newcommand{\Sc}{\mathcal{S}}
\newcommand{\Sr}{\mathscr{S}}
\newcommand{\U}{\mathscr{U}}
\newcommand{\X}{\mathscr{X}}
\newcommand{\Y}{\mathscr{Y}}
\newcommand{\Z}{\mathbb{Z}}
\newcommand{\ve}[1]{\boldsymbol{#1}}

\begin{document}

\title{Knot points of typical continuous functions}

\author{David Preiss}
\address{Mathematics Institute, University of Warwick, Coventry CV4 7AL, United Kingdom}
\email{d.preiss@warwick.ac.uk}

\author{Shingo Saito}
\address{Institute of Mathematics for Industry, Kyushu University, 744, Motooka, Nishi-ku, Fukuoka, 819-0395, Japan}
\email{ssaito@imi.kyushu-u.ac.jp}


\date{}

\dedicatory{}

\begin{abstract}
It is well known that
most continuous functions are nowhere differentiable.
Furthermore, in terms of Dini derivatives,
most continuous functions are nondifferentiable in the strongest possible sense
except in a small set of points.
In this paper, we completely characterise families $\Sc$ of sets of points
for which most continuous functions have the property that
such small set of points belongs to $\Sc$.
The proof uses a topological zero-one law and the Banach-Mazur game.
\end{abstract}

\maketitle

\section{Introduction}
Since Banach~\cite{Banach} and Mazurkiewicz~\cite{Mazurkiewicz}
independently proved in 1931 that
most continuous functions are nowhere differentiable,
many mathematicians have been investigating
properties of most functions.
In the study of most functions,
we first have to make clear
what `most' means.
Although a number of definitions have been invented,
we shall use the most classical notion,
upon which the above-mentioned papers by Banach and Mazurkiewicz are based.
Let us begin by recalling the classical notion of Baire category.

We write $I$ for the unit interval $[0,1]$,
and $C(I)$ for the set of all continuous functions from $I$ to $\R$.
The space $C(I)$ is a Banach space
under the supremum norm~$\lVert\cdot\rVert$. 
Recall that a subset $A$ of a topological space $X$
is \emph{nowhere dense} if $\Int\bar{A}=\emptyset$;
it is \emph{meagre} if it can be expressed
as a countable union of nowhere dense subsets of $X$;
it is \emph{residual} (or \emph{comeagre}) if $A^c$ is meagre.
Properties of `most' functions will be understood as
those possessed by all functions in a residual subset of $C(I)$:

\begin{definition}
 We say that a \emph{typical} (or \emph{generic}) function $f\in C(I)$
 has a property $P$
 if the set of all $f\in C(I)$ with the property $P$ is residual in $C(I)$.
\end{definition}

As mentioned at the beginning,
a typical function is nowhere differentiable,
so its derivative cannot be considered.
In place of its derivative,
we shall look at its \emph{Dini derivatives}:

\begin{definition}\label{definition:Dini_derivatives}
 Let $f\in C(I)$. We define
 \[
  D^+f(x)=\limsup_{y\downarrow x}\frac{f(y)-f(x)}{y-x},\qquad
  D_+f(x)=\liminf_{y\downarrow x}\frac{f(y)-f(x)}{y-x}
 \]
 for $x\in[0,1)$, and
 \[
  D^-f(x)=\limsup_{y\uparrow x}\frac{f(y)-f(x)}{y-x},\qquad
  D_-f(x)=\liminf_{y\uparrow x}\frac{f(y)-f(x)}{y-x}
 \]
 for $x\in(0,1]$.
 They are called the \emph{Dini derivatives} of $f$ at $x$.
\end{definition}

The oldest result about the behaviour of the Dini derivatives of a typical continuous function
is the following theorem by Jarn\'{\i}k~\cite{Jarnik}:

\begin{theorem}[\cite{Jarnik}]\label{theorem:Jarnik}
 A typical function $f\in C(I)$ has the property that
 \[
  D^+f(x)=D^-f(x)=\infty,\qquad
  D_+f(x)=D_-f(x)=-\infty
 \]
 for almost every $x\in(0,1)$.
\end{theorem}

This theorem leads us to the following definition:

\begin{definition}\label{definition:knot_point}
 We say that a point $x\in I$ is a \emph{knot point} of $f\in C(I)$ if
 \begin{itemize}
  \item $x\in(0,1)$, $D^+f(x)=D^-f(x)=\infty$, and $D_+f(x)=D_-f(x)=-\infty$; or
  \item $x=0$, $D^+f(x)=\infty$, and $D_+f(x)=-\infty$; or
  \item $x=1$, $D^-f(x)=\infty$, and $D_-f(x)=-\infty$.
 \end{itemize}
 For $f\in C(I)$, we write $N(f)$ for the set of all
 points in $I$ that are \emph{not} knot points of $f$.
\end{definition}

Theorem~\ref{theorem:Jarnik} means that
a typical function $f\in C(I)$ has the property that $N(f)$ is Lebesgue null,
i.e.\ small from the measure-theoretic viewpoint.
It is natural to ask in what sense of smallness
it is true that a typical function has the property that $N(f)$ is small.
Zaj\'{\i}\v{c}ek and the first author answered this question
in unpublished work~\cite{PreissZajicek} by giving
a necessary and sufficient condition for a $\sigma$-ideal $\I$
to satisfy that a typical function $f\in C(I)$ has the property that $N(f)\in\I$
(see Theorem~\ref{theorem:PreissZajicek} for the precise statement).
The purpose of this paper is to generalise this theorem
by giving a necessary and sufficient condition
for an \emph{arbitrary} family $\Sc$ of subsets of $I$
to satisfy that a typical function $f\in C(I)$ has the property that $N(f)\in\Sc$
(see Theorem~\ref{theorem:maintheorem_general} for the precise statement).

The paper is structured as follows.
We first state the main theorem in Section~2.
Section~3 gives basic properties of sets and functions
that will be used later in this paper.
Section~4 provides descriptive set-theoretic arguments
and reduces the main theorem to what we call the key proposition.
In Section~5 we prove the key proposition using the Banach-Mazur game.

\begin{remark}
 The results in this paper are part of the second author's PhD thesis~\cite{thesis}.
\end{remark}

\section{Statement of the main theorem}
\subsection{Residuality of families of $F_{\sigma}$ sets}
In order to state the main theorem,
we need the definition of residuality of families of $F_{\sigma}$ sets,
given in \cite{Saito}.

We write $\K$ for the set of all closed subsets of $I$,
and equip it with the Hausdorff metric $d$,
where we define $d(K,\emptyset)=1$ for any nonempty set $K\in\K$.
Excluding $0$ from the set $\N$ of all positive integers,
we denote by $\K^{\N}$ the set of all sequences of members of $\K$,
and by $\K_{\nearrow}^{\N}$ the subset of $\K^{\N}$ consisting of
all increasing sequences:
\begin{align*}
 \K^{\N}&=\{(K_n)\mid\text{$K_n\in\K$ for all $n\in\N$}\},\\
 \K_{\nearrow}^{\N}&=\{(K_n)\in\K^{\N}\mid K_1\subset K_2\subset\cdots\}.
\end{align*}
The spaces $\K$, $\K^{\N}$, and $\K_{\nearrow}^{\N}$ are
all compact metrisable topological spaces.

We write $\F_{\sigma}$ for the family of all $F_{\sigma}$ subsets of $I$.
The following is the main theorem of \cite{Saito}:

\begin{theorem}[\cite{Saito}]\label{theorem:residual_equivalent}
 For a subfamily $\F$ of $\F_{\sigma}$,
 the following are equivalent:
 \begin{enumerate}
  \item $\{(K_n)\in\K^{\N}\mid\bigcup_{n=1}^{\infty}K_n\in\F\}$
   is residual in $\K^{\N}$;
  \item $\{(K_n)\in\K_{\nearrow}^{\N}\mid\bigcup_{n=1}^{\infty}K_n\in\F\}$
   is residual in $\K_{\nearrow}^{\N}$.
 \end{enumerate}
\end{theorem}

\begin{definition}
 A subfamily $\F$ of $\F_{\sigma}$ is said to be \emph{residual}
 if the conditions in Theorem~\ref{theorem:residual_equivalent} hold.
\end{definition}

\begin{proposition}[{\cite[Proposition~1.5]{Saito}}]\label{prop:sigma_ideal_equivalency_residuality}
 If $\I$ is a $\sigma$-ideal on $I$, then
 $\I\cap\K$ is residual in $\K$ if and only if
 $\I\cap\F_{\sigma}$ is residual in $\F_{\sigma}$.
\end{proposition}

\subsection{Statement of the main theorem}
We are now ready to state the main theorem of this paper.
The following theorem has been announced by Zaj\'{\i}\v{c}ek~\cite{Zajicek}
and proved by Zaj\'{\i}\v{c}ek and the first author~\cite{PreissZajicek}:

\begin{theorem}[\cite{PreissZajicek}, {\cite[Theorem~2.5]{Zajicek}}]%
 \label{theorem:PreissZajicek}
 For a $\sigma$-ideal $\I$ on $I$, the following are equivalent:
 \begin{enumerate}
  \item a typical function $f\in C(I)$ has the property that $N(f)\in\I$;
  \item $\I\cap\K$ is residual in $\K$.
 \end{enumerate}
\end{theorem}

Our main theorem is the following:

\begin{theorem}[Main Theorem]\label{theorem:maintheorem_general}
 For a family $\Sc$ of subsets of $I$, the following are equivalent:
 \begin{enumerate}
  \item a typical function $f\in C(I)$ has the property that $N(f)\in\Sc$.
  \item $\Sc\cap\F_{\sigma}$ is residual in $\F_{\sigma}$.
 \end{enumerate}
\end{theorem}

Observe that Theorem~\ref{theorem:maintheorem_general} generalises
Theorem~\ref{theorem:PreissZajicek} due to
Proposition~\ref{prop:sigma_ideal_equivalency_residuality}.

\section{Basic properties of $\K$ and $N(f,a)$}
\subsection{Basic properties of $\K$}
For $a\in I$, $A\subset I$, and $r>0$, we set
\begin{align*}
 B(a,r)&=\{x\in I\mid \lvert x-a\rvert<r\},&
 \bar{B}(a,r)&=\{x\in I\mid \lvert x-a\rvert\le r\},\\
 B(A,r)&=\bigcup_{a\in A}B(a,r),&
 \bar{B}(A,r)&=\bigcup_{a\in A}\bar{B}(a,r).
\end{align*}

\begin{lemma}
 If $K,L\in\K$ and $r>0$ are such that $K\subset B(L,r)$,
 then $K\subset B(L,r-\varepsilon)$ for some $\varepsilon>0$.
\end{lemma}

\begin{proof}
 Suppose that $K\not\subset B(L,r-\varepsilon)$ for all $\varepsilon>0$,
 and take $x_n\in K\setminus B(L,r-1/n)$ for each $n\in\N$.
 We may assume that $x_n$ is convergent, say to $x$.
 Since $x\in K\subset B(L,r)$, there exists $y\in L$ with $\lvert x-y\rvert<r$.
 By the choice of $x_n$, we have $\lvert x_n-y\rvert\ge r-1/n$,
 and so $\lvert x-y\rvert\ge r$, which is a contradiction.
\end{proof}

\begin{corollary}
 For every $r>0$,
 the set $\{(K,L)\in\K^2\mid K\subset B(L,r)\}$ is open in $\K^2$.
\end{corollary}

\begin{proof}
 Let $(K_0,L_0)$ belong to the set, and
 take $\varepsilon>0$ with $K_0\subset B(L_0,r-\varepsilon)$ using the previous lemma.
 If $(K,L)\in\K^2$ satisfies $d(K,K_0)<\varepsilon/2$ and $d(L,L_0)<\varepsilon/2$,
 then
 \[
  K\subset B(K_0,\varepsilon/2)\subset B(L_0,r-\varepsilon/2)\subset B(L,r).
 \]
 This completes the proof.
\end{proof}

\subsection{Definition of $N(f,a)$}
\begin{definition}\label{definition:N(f,a)}
 For $f\in C(I)$ and $a>0$, we define
 \begin{align*}
  N^+(f,a)
  &=\{x\in[0,1-2^{-a}]\mid\text{$f(y)-f(x)\le a(y-x)$ for all $y\in[x,x+2^{-a}]$}\},\\
  N_+(f,a)
  &=\{x\in[0,1-2^{-a}]\mid\text{$f(y)-f(x)\ge-a(y-x)$ for all $y\in[x,x+2^{-a}]$}\},\\
  N^-(f,a)
  &=\{x\in[2^{-a},1]\mid\text{$f(y)-f(x)\ge a(y-x)$ for all $y\in[x-2^{-a},x]$}\},\\
  N_-(f,a)
  &=\{x\in[2^{-a},1]\mid\text{$f(y)-f(x)\le-a(y-x)$ for all $y\in[x-2^{-a},x]$}\},
 \end{align*}
 and
 \begin{align*}
  \hat{N}(f,a)&=N^+(f,a)\cup N_-(f,a),\\
  \check{N}(f,a)&=N_+(f,a)\cup N^-(f,a),\\
  N(f,a)&=\hat{N}(f,a)\cup\check{N}(f,a)\\
        &=N^+(f,a)\cup N_+(f,a)\cup N^-(f,a)\cup N_-(f,a).
 \end{align*}
\end{definition}

\begin{convention}\label{conv:tilde}
 We shall use the symbol $\tilde{N}$ in a statement
 to mean that the statement with $\tilde{N}$ replaced by $\hat{N}$
 and the statement with $\tilde{N}$ replaced by $\check{N}$
 are both true;
 for instance, by $\tilde{N}(f,a)\subset\tilde{N}(g,b)$ we mean
 $\hat{N}(f,a)\subset\hat{N}(g,b)$ and $\check{N}(f,b)\subset\check{N}(g,b)$.
\end{convention}

\begin{remark}
 The mean value theorem shows that
 \[
  \lvert2^{-a}-2^{-b}\rvert\le\lvert a-b\rvert\log2\le\lvert a-b\rvert
 \]
 for all $a,b>0$.
 This estimate will sometimes be used implicitly in this paper.
\end{remark}

\begin{proposition}
 If $f\in C(I)$ and $0<a_1<a_2<\dots\to\infty$,
 then $N(f)=\bigcup_{n=1}^{\infty}N(f,a_n)$.
\end{proposition}

\begin{proof}
 Trivial.
\end{proof}

\subsection{Descriptive properties of knot points}
\begin{proposition}\label{prop:N(f)_Fsigma}
 For every $f\in C(I)$ and $a>0$, the sets
 $N^{\pm}(f,a)$, $N_{\pm}(f,a)$, $\tilde{N}(f,a)$, and $N(f,a)$
 are all closed.
 Therefore $N(f)$ is $F_{\sigma}$ for every $f\in C(I)$.
\end{proposition}

\begin{proof}
 Obviously it suffices to show that $N^+(f,a)$ is closed.
 Suppose that a sequence $x_n$ of points in $N^+(f,a)$ converges to a point $x$.
 Since $x_n\in[0,1-2^{-a}]$ for all $n\in\N$, we have $x\in[0,1-2^{-a}]$.
 Assume for a contradiction that
 $f(y)-f(x)>a(y-x)$ for some $y\in[x,x+2^{-a}]$.
 By the continuity of $f$, we may assume that $y\in(x,x+2^{-a})$.
 Then since $x_n$ converges to $x$ and $f$ is continuous,
 there exists $n\in\N$ such that $y\in(x_n,x_n+2^{-a})$ and $f(y)-f(x_n)>a(y-x_n)$,
 which contradicts $x_n$ belonging to $N^+(f,a)$.
\end{proof}

By Proposition~\ref{prop:N(f)_Fsigma},
we can restate our main theorem (Theorem~\ref{theorem:maintheorem_general})
as follows:

\begin{theorem}[Main Theorem]\label{theorem:maintheorem2}
 For a subfamily $\F$ of $\F_{\sigma}$, the following are equivalent:
 \begin{enumerate}
  \item a typical function $f\in C(I)$ has the property that $N(f)\in\F$;
  \item $\F$ is residual.
 \end{enumerate}
\end{theorem}

\subsection{Continuity of $N(f,a)$}
\begin{proposition}\label{prop:continuity}
 Suppose that $0<a<b$ and $\varepsilon>0$.
 Then there exists $\delta>0$ such that
 whenever $f,g\in C(I)$ satisfy $\lVert f-g\rVert<\delta$, we have
 $\tilde{N}(f,a)\subset B\bigl(\tilde{N}(g,b),\varepsilon\bigr)$ and
 $N(f,a)\subset B\bigl(N(g,b),\varepsilon\bigr)$.
\end{proposition}

\begin{proof}
 We may assume that $\varepsilon<2^{-a}-2^{-b}$ without loss of generality.
 Choose $\delta>0$ with $\delta<\varepsilon(b-a)/2$.
 We shall show that this $\delta$ satisfies the required condition.
 It suffices to prove that $N^+(f,a)\subset B\bigl(N^+(g,b),\varepsilon\bigr)$.

 Take any $x\in N^+(f,a)$, and let $y_0\in[x,x+2^{-a}]$ be a point
 at which the continuous function $y\longmapsto g(y)-by$
 defined on $[x,x+2^{-a}]$ attains its maximum.
 It is enough to show that $x\le y_0<x+\varepsilon$ and
 $y_0\in N^+(g,b)$.

 The definition of $y_0$ gives $g(y_0)-by_0\ge g(x)-bx$, which implies
 \[
  b(y_0-x)\le g(y_0)-g(x)<f(y_0)-f(x)+2\delta\le a(y_0-x)+2\delta
 \]
 because $x\in N^+(f,a)$ and $y_0\in[x,x+2^{-a}]$.
 It follows that $y_0-x<2\delta/(b-a)<\varepsilon$.

 With the aim of proving $y_0\in N^+(g,b)$, take any $y\in[y_0,y_0+2^{-b}]$.
 Since
 \[
  x\le y_0\le y\le y_0+2^{-b}<x+\varepsilon+2^{-b}<x+2^{-a},
 \]
 the definition of $y_0$ again gives
 $g(y_0)-by_0\ge g(y)-by$,
 or equivalently $g(y)-g(y_0)\le b(y-y_0)$.
 This completes the proof.
\end{proof}

\subsection{Properties of continuously differentiable functions}
\begin{lemma}\label{lem:C1_continuity_single_function}
 If $f\in C^1(I)$ and $0<a<b$, then there exists $\delta>0$ such that
 $B\bigl(\tilde{N}(f,a),\delta\bigr)\subset\tilde{N}(f,b)$.
\end{lemma}

\begin{proof}
 By symmetry, it suffices to show that
 $B\bigl(N^+(f,a),\delta\bigr)\subset N^+(f,b)$
 for some $\delta>0$.
 Suppose that this is false.
 For each $n\in\N$,
 let $\delta_n=(2^{-a}-2^{-b})/n$ and
 take $x_n\in B\bigl(N^+(f,a),\delta_n\bigr)\setminus N^+(f,b)$.
 We may assume that $x_n$ converges, say to $x$.
 Observe that
 \[
  x\in\bigcap_{n=1}^{\infty}B\bigl(N^+(f,a),\delta_n+\lvert x-x_n\rvert\bigr)
   =N^+(f,a).
 \]

 Since
 \begin{align*}
  x_n&\in B\bigl(N^+(f,a),\delta_n\bigr)\setminus N^+(f,b)\\
     &\subset B([0,1-2^{-a}],2^{-a}-2^{-b})\setminus N^+(f,b)\\
     &\subset[0,1-2^{-b}]\setminus N^+(f,b),
 \end{align*}
 we may take $y_n\in(x_n,x_n+2^{-b}]$ with
 $f(y_n)-f(x_n)>b(y_n-x_n)$.
 We may assume that $y_n$ converges, say to $y$.
 The continuity of $f$ shows that $f(y)-f(x)\ge b(y-x)$,
 whereas we have $f(y)-f(x)\le a(y-x)$ because
 $x\in N^+(f,a)$ and $x\le y\le x+2^{-b}<x+2^{-a}$.
 It follows that $y=x$.

 By the mean value theorem, we may take $z_n\in(x_n,y_n)$ with
 \[
  f'(z_n)=\frac{f(y_n)-f(x_n)}{y_n-x_n}>b.
 \]
 Since both $x_n$ and $y_n$ converge to $x$, so does $z_n$.
 The continuity of $f'$ shows that $f'(x)\ge b$,
 which contradicts $x\in N^+(f,a)$.
\end{proof}

\begin{corollary}\label{cor:C1_continuity_single_interior}
 If $f\in C^1(I)$ and $0<a<b$, then
 $\tilde{N}(f,a)\subset\Int\tilde{N}(f,b)$.
\end{corollary}

\begin{proof}
 Immediate from Lemma~\ref{lem:C1_continuity_single_function}.
\end{proof}

\begin{proposition}\label{prop:C1_continuity}
 Suppose that $f\in C^1(I)$ and $0<a<b$.
 Then there exists $\delta>0$ such that
 $B\bigl(\tilde{N}(g,a),\delta\bigr)\subset\tilde{N}(f,b)$
 for every $g\in B(f,\delta)$.
\end{proposition}

\begin{proof}
 Set $c=(a+b)/2$, so that $0<a<c<b$.
 By Lemma~\ref{lem:C1_continuity_single_function}
 we may find $\varepsilon>0$
 with $B\bigl(\tilde{N}(f,c),2\varepsilon\bigr)\subset\tilde{N}(f,b)$,
 and by Proposition~\ref{prop:continuity}
 we may find $\tau>0$ such that
 $\tilde{N}(g,a)\subset B\bigl(\tilde{N}(f,c),\varepsilon\bigr)$
 for all $g\in B(f,\tau)$.
 We set $\delta=\min\{\varepsilon,\tau\}$.
 Then for every $g\in B(f,\delta)$, we have
 \[
  B\bigl(\tilde{N}(g,a),\delta\bigr)
  \subset B\bigl(\tilde{N}(f,c),\delta+\varepsilon\bigr)
  \subset B\bigl(\tilde{N}(f,c),2\varepsilon\bigr)
  \subset\tilde{N}(f,b).\qedhere
 \]
\end{proof}

\begin{lemma}\label{lem:C1_open_interval}
 Suppose that $f\in C^1(I)$ and $0<a<c<b$.
 Then there exists $\varepsilon>0$ such that
 for each $x\in[0,1-2^{-a}]\setminus N^+(f,b)$,
 we may find $y\in(x+\varepsilon,x+2^{-b}]$
 with $f(y)-f(x)>c(y-x)$.
\end{lemma}

\begin{proof}
 Suppose that the lemma is false.
 Then for each $n\in\N$, we may find $x_n\in[0,1-2^{-a}]\setminus N^+(f,b)$
 such that $f(y)-f(x_n)\le c(y-x_n)$ for all $y\in(x_n+1/n,x_n+2^{-b}]$.
 We may assume that $x_n$ converges, say to $x\in[0,1-2^{-a}]\subset[0,1-2^{-b}]$.

 Firstly, we prove that
 $f(y)-f(x)\le c(y-x)$ for all $y\in(x,x+2^{-b})$.
 Fix such $y$.
 For sufficiently large $n$, since $y\in(x_n+1/n,x_n+2^{-b})$,
 we have $f(y)-f(x_n)\le c(y-x_n)$ by the choice of $x_n$.
 Letting $n\to\infty$, we obtain $f(y)-f(x)\le c(y-x)$.

 Now, it follows that $f'(x)\le c$, and so $f'\le b$ in some neighbourhood of $x$
 because $f\in C^1(I)$.
 Take $n\in\N$ so large that the interval $[x_n,x_n+1/n]$
 is contained in the neighbourhood.
 Then the mean value theorem shows that
 $f(y)-f(x_n)\le b(y-x_n)$ for all $y\in[x_n,x_n+1/n]$.
 This, together with the choice of $x_n$, implies that
 $x_n\in N^+(f,b)$, a contradiction.
\end{proof}

\begin{proposition}\label{prop:C1_open_interval}
 Suppose that $f\in C^1(I)$ and $0<a<b$.
 Then there exists $l>0$ such that every set of one of the following forms
 contains an open interval of length $l$:
 \begin{enumerate}
  \item $\{y\in[x,x+2^{-a}]\mid f(y)-f(x)>a(y-x)\}$
        for $x\in[0,1-2^{-a}]\setminus N^+(f,b)$;
  \item $\{y\in[x,x+2^{-a}]\mid f(y)-f(x)<-a(y-x)\}$
        for $x\in[0,1-2^{-a}]\setminus N_+(f,b)$;
  \item $\{y\in[x-2^{-a},x]\mid f(y)-f(x)<a(y-x)\}$
        for $x\in[2^{-a},1]\setminus N^-(f,b)$;
  \item $\{y\in[x-2^{-a},x]\mid f(y)-f(x)>-a(y-x)\}$
        for $x\in[2^{-a},1]\setminus N_-(f,b)$.
 \end{enumerate}
\end{proposition}

\begin{proof}
 Set $c=(a+b)/2$ and choose $\varepsilon>0$ as in Lemma~\ref{lem:C1_open_interval}.
 Then take $l>0$ so that $l/2<\min\{\varepsilon,2^{-a}-2^{-b}\}$
 and $(\lVert f'\rVert+a)l/2<(c-a)\varepsilon$.
 We shall show that this $l$ satisfies the required condition.
 By symmetry, we only need to look at sets of the first form.

 Let $x\in[0,1-2^{-a}]\setminus N^+(f,b)$ and set
 \[
  S=\{y\in[x,x+2^{-a}]\mid f(y)-f(x)>a(y-x)\}.
 \]
 By the choice of $\varepsilon$, we may find $t\in(x+\varepsilon,x+2^{-b}]$
 with $f(t)-f(x)>c(t-x)$.
 It suffices to show that $S$ contains the open interval $(t-l/2,t+l/2)$.
 If $y\in(t-l/2,t+l/2)$, then since
 \begin{align*}
  y&>t-l/2>x+\varepsilon-l/2>x,\\
  y&<t+l/2\le x+2^{-b}+l/2<x+2^{-a},
 \end{align*}
 and
 \begin{align*}
  f(y)-f(x)-a(y-x)
  &=\bigl(f(y)-f(t)\bigr)+\bigl(f(t)-f(x)\bigr)-a(y-x)\\
  &>-\lVert f'\rVert\lvert y-t\rvert+c(t-x)-a(y-x)\\
  &=(c-a)(t-x)-\lVert f'\rVert\lvert y-t\rvert-a(y-t)\\
  &\ge(c-a)(t-x)-(\lVert f'\rVert+a)\lvert y-t\rvert\\
  &>(c-a)\varepsilon-(\lVert f'\rVert+a)l/2\\
  &>0,
 \end{align*}
 it follows that $y\in S$.
\end{proof}

\subsection{Bump functions}
\begin{definition}\label{definition:bump}
 Let $\hat{H}$ and $\check{H}$ be disjoint finite subsets of $I$,
 and $h$ and $w$ be positive numbers.
 A \emph{bump function} of height $h$ and width $w$
 located at $\hat{H}$ and $\check{H}$ is a function
 $\varphi\in C^1(I)$ with the following properties:
 \begin{itemize}
  \item $\lVert\varphi\rVert=h$;
  \item $\varphi(x)=h$ for all $x\in\hat{H}$ and
        $\varphi(x)=-h$ for all $x\in\check{H}$;
  \item $\{x\in I\mid\varphi(x)>0\}\subset B(\hat{H},w)$ and
        $\{x\in I\mid\varphi(x)<0\}\subset B(\check{H},w)$.
 \end{itemize}
\end{definition}

\begin{remark}
 If $\hat{H}$, $\check{H}$, $h$, and $w$ satisfy the conditions
 at the beginning of the definition above,
 there exists a bump function of height $h$ and width $w$
 located at $\hat{H}$ and $\check{H}$.
\end{remark}

\begin{proposition}\label{prop:bump_easy}
 Let $f\in C(I)$ and $a>0$.
 Suppose that $\varphi$ is a bump function of height $h>0$ and width $w>0$
 located at $\hat{H}$ and $\check{H}$,
 where $\hat{H}$ and $\check{H}$ are disjoint finite subsets of $I$.
 Then, setting $g=f+\varphi$, we have
 $\tilde{H}\cap\tilde{N}(f,a)\subset\tilde{N}(g,a)$.
\end{proposition}

\begin{proof}
 It suffices to show that $\hat{H}\cap\hat{N}(f,a)\subset\hat{N}(g,a)$.
 Let $x\in\hat{H}\cap\hat{N}(f,a)$.
 Then $x\in N^+(f,a)\cup N_-(f,a)$, and we may assume that
 $x\in N^+(f,a)$ by symmetry.
 We have $x\in[0,1-2^{-a}]$ by the definition of $N^+(f,a)$;
 if $y\in[x,x+2^{-a}]$, then
 \[
  g(y)-g(x)
  =\bigl(f(y)+\varphi(y)\bigr)-\bigl(f(x)+h\bigr)
  \le f(y)-f(x)
  \le a(y-x).
 \]
 It follows that $x\in N^+(g,a)$.
\end{proof}

\begin{proposition}\label{prop:bump_difficult}
 Suppose that $f\in C^1(I)$, $0<a<b$, and $h>0$.
 Then there exists $\mu>0$ with the following property:
 \begin{quotation}
  Suppose that $\varphi$ is a bump function of height $h$ and width $w>0$
  located at $\hat{H}$ and $\check{H}$, where $\hat{H}$ and $\check{H}$
  are disjoint finite subsets of $I$
  satisfying $B(\tilde{H},\mu)=I$.
  Then, setting $g=f+\varphi$, we have
  $\tilde{N}(g,a)\subset\tilde{N}(f,b)\cap B(\tilde{H},w)$.
 \end{quotation}
\end{proposition}

\begin{proof}
 Choose $l>0$ as in Proposition~\ref{prop:C1_open_interval}.
 Take $\mu>0$ so small that $\mu<l/2$, $2\mu<2^{-a}$,
 and $2\mu(\lVert f'\rVert+a)<h$.
 We shall show that this $\mu$ satisfies the required condition.
 Let $\varphi$ and $g$ be as in the statement.
 By symmetry, it suffices to show that
 $N^+(g,a)\subset N^+(f,b)\cap B(\hat{H},w)$.
 Let $x\in N^+(g,a)$.

 Firstly, we show that $x\in N^+(f,b)$.
 Assume, to derive a contradiction, that $x\notin N^+(f,b)$.
 Then, since
 \[
  x\in N^+(g,a)\setminus N^+(f,b)\subset[0,1-2^{-a}]\setminus N^+(f,b),
 \]
 the set $\{y\in[x,x+2^{-a}]\mid f(y)-f(x)>a(y-x)\}$ contains an
 open interval of length $l$.
 Because $B(\hat{H},l/2)\supset B(\hat{H},\mu)=I$,
 we may find $y\in\hat{H}$ such that
 $y\in[x,x+2^{-a}]$ and $f(y)-f(x)>a(y-x)$.
 Then
 \[
  g(y)-g(x)=\bigl(f(y)+h\bigr)-\bigl(f(x)+\varphi(x)\bigr)
  \ge f(y)-f(x)>a(y-x),
 \]
 which contradicts the assumption that $x\in N^+(g,a)$.

 Secondly, we show that $x\in B(\hat{H},w)$.
 Because $B(\hat{H},\mu)=I$, we may find $y\in[x,x+2\mu]\cap\hat{H}$.
 Then
 \begin{align*}
  a(y-x)&\ge g(y)-g(x)=\bigl(f(y)+h\bigr)-\bigl(f(x)+\varphi(x)\bigr)\\
        &\ge h-\varphi(x)-\lVert f'\rVert(y-x),
 \end{align*}
 which implies that
 \[
  \varphi(x)\ge h-(\lVert f'\rVert+a)(y-x)\ge h-2\mu(\lVert f'\rVert+a)>0.
 \]
 It follows that $x\in B(\hat{H},w)$.
\end{proof}

\begin{definition}\label{definition:mu}
 If $f\in C^1(I)$, $0<a<b$ and $h>0$,
 then $\mu(f,a,b,h)$ denotes a positive number $\mu$
 with the property in Proposition~\ref{prop:bump_difficult}.
\end{definition}

\section{A topological zero-one law and a key proposition}
This section uses some terminology and concepts in descriptive set theory;
see \cite{Kechris} for details.

\subsection{A topological zero-one law}
\begin{convention}
 We shall use boldface letters to denote sequences, and
 denote a term of a sequence by the corresponding normal letter
 accompanied with a subscript.
 For example, the $n$th term of a sequence $\ve{x}$ is $x_n$.
\end{convention}

\begin{definition}
 Let $X$ be a set.
 A subset $A$ of $X^{\N}$ is said to be \emph{invariant under finite permutations}
 if for every permutation $\sigma$ on $\N$ that fixes all but finitely many
 positive integers and for every $\ve{x}\in A$,
 we have $(x_{\sigma(n)})\in A$.
\end{definition}

\begin{proposition}[{{\cite[Theorem~8.46]{Kechris}}}]\label{prop:zero-one-Kechris}
 Let $X$ be a Baire space and $G$ a group of homeomorphisms on $X$
 with the property that for every pair of nonempty open subsets $U$ and $V$ of $X$,
 there exists $\varphi\in G$ such that $\varphi(U)\cap V\ne\emptyset$.
 Suppose that a subset $A$ of $X$ has the Baire property
 and that $\varphi(A)=A$ for every $\varphi\in G$.
 Then $A$ is either meagre or residual.
\end{proposition}

\begin{remark}
 If $G$ is a group of bijections on a set $X$ and $A$ is a subset of $X$,
 then the condition that $\varphi(A)=A$ for all $\varphi\in G$
 is equivalent to the condition that $\varphi(A)\subset A$ for all $\varphi\in G$.
\end{remark}

For $n\in\N$, set $[n]=\{1,\dotsc,n\}$.

\begin{proposition}\label{prop:zero-one}
 Let $X$ be a Baire space and $A$ a subset of $X^{\N}$
 that is invariant under finite permutations and has the Baire property.
 Then $A$ is either meagre or residual.
\end{proposition}

\begin{proof}
 Since the proposition is obvious if $X=\emptyset$,
 we may assume that $X\ne\emptyset$ and take an element $a\in X$.

 For each permutation $\sigma$ on $\N$,
 let $\varphi_{\sigma}$ be the homeomorphism on $X^{\N}$
 defined by $\varphi_{\sigma}(\ve{x})=(x_{\sigma(n)})$ for $\ve{x}\in X^{\N}$.
 Write $G$ for the set of all $\varphi_{\sigma}$
 where $\sigma$ is a permutation that fixes all but finitely many positive integers.
 It is obvious that $G$ is a group.
 In the light of Proposition~\ref{prop:zero-one-Kechris},
 it suffices to show that
 for every pair of nonempty open subsets $U$ and $V$ of $X^{\N}$,
 there exists $\varphi\in G$ such that $\varphi(U)\cap V\ne\emptyset$.

 Let $U$ and $V$ be nonempty open subsets of $X^{\N}$.
 Take $\ve{u}\in U$ and $\ve{v}\in V$,
 and choose $m\in\N$
 so that $\ve{x}\in U$ if $x_n=u_n$ for all $n\in[m]$,
 and $\ve{x}\in V$ if $x_n=v_n$ for all $n\in[m]$.
 Define a permutation $\sigma$ on $\N$ by setting
 \[
  \sigma(n)=
  \begin{cases}
   n+m&\text{for $n\in[m]$};\\
   n-m&\text{for $n\in[2m]\setminus[m]$};\\
   n&\text{for $n\in\N\setminus[2m]$}.
  \end{cases}
 \]
 Then $\sigma$ fixes all integers greater than $2m$,
 and so $\varphi_{\sigma}\in G$.
 Moreover, $\varphi_{\sigma}$ satisfies
 $\varphi_{\sigma}(U)\cap V\ne\emptyset$ because
 $(u_1,\dotsc,u_m,v_1,\dotsc,v_m,a,a,\dotsc)\in U$ and
 \[
  \varphi_{\sigma}\bigl((u_1,\dotsc,u_m,v_1,\dotsc,v_m,a,a,\dotsc)\bigr)
  =(v_1,\dotsc,v_m,u_1,\dotsc,u_m,a,a,\dotsc)\in V.
 \]
 This completes the proof.
\end{proof}

\subsection{Definition and basic properties of $\X$}
\begin{convention}
 Because the complexity of the discussion below
 forces us to use many indices,
 we shall often use superscripts as well as subscripts
 to denote indices rather than exponents.
 We do use powers occasionally, but
 the meaning will always be clear from the context.
\end{convention}

Write $\Z_+$ for the set of all nonnegative integers:
$\Z_+=\{0,1,2,\dots\}=\{0\}\cup\N$.

\begin{definition}\label{definition:X}
 \begin{enumerate}
  \item We put
        \begin{align*}
         X&=\{\ve{a}\in(0,\infty)^{\N}\mid a_1<a_2<\dotsb\to\infty\},\\
         Y&=\{\ve{\delta}\in(0,1)^{\N}\mid \delta_1>\delta_2>\dotsb\to0\},\\
         Z&=\{\ve{n}\in\N^{\N}\mid\text{$n_{j+1}\ge n_j+j$ for all $j\in\N$}\}.
        \end{align*}
        These are Polish spaces in the relative topology because
        they are $G_{\delta}$ subsets of the Polish spaces
        $(0,\infty)^{\N}$, $(0,1)^{\N}$, and $\N^{\N}$ respectively.
  \item For $\ve{n}\in Z$ and $j,m\in\N$ with $j\le m$, we define
        a finite subset $A_j^m(\ve{n})$ of $\N$ by
        \[
         A_j^m(\ve{n})=[n_j]\cup\bigcup_{i=j}^{m-1}\{n_i+1,\dotsc,n_i+j-1\}.
        \]
        For $\ve{n}\in Z$ and $k\in\Z_+$, we define $\ve{n}^k\in Z$
        by setting $n_j^k=n_{j+k}$ for $j\in\N$.
  \item Let $\ve{n}\in Z$ and $\ve{\delta}\in Y$.
        For $k\in\Z_+$, we define $\Sr_k(\ve{n},\ve{\delta})$ as the set of
        all $\ve{K}\in\K^{\N}$ such that
        \[
         \bigcup_{n\in A_j^m(\ve{n}^k)\setminus A_j^{m-1}(\ve{n}^k)}K_n\subset
         \bigcup_{n\in A_{j-1}^{m-1}(\ve{n}^k)}B(K_n,\delta_m)
        \]
        whenever $2\le j\le m-1$.
        In addition we define
        $\Sr(\ve{n},\ve{\delta})=\bigcup_{k=0}^{\infty}\Sr_k(\ve{n},\ve{\delta})$.
 \end{enumerate}
\end{definition}

\begin{remark}
 To be precise, the definition of $A_j^m(\ve{n})$ is as follows:
 \[
  A_j^m(\ve{n})=
  \begin{cases}
   [n_j]&\text{if $j=1$ or $j=m$};\\
   [n_j]\cup\bigcup_{i=j}^{m-1}\{n_i+1,\dotsc,n_i+j-1\}&\text{if $2\le j\le m-1$}.
  \end{cases}
 \]
\end{remark}

\begin{remark}
 For the reader's convenience,
 we spell out $A_j^m(\ve{n})$ for small $j$ and $m$,
 writing $A_j^m=A_j^m(\ve{n})$ for simplicity:
 \begin{enumerate}
  \item if $j=1$, then $A_1^m=[n_1]$ for all $m\in\N$;
  \item if $j=2$, then $A_2^2=[n_2]$, $A_2^3=[n_2+1]$,
        $A_2^4=[n_2+1]\cup\{n_3+1\}$, $A_2^5=[n_2+1]\cup\{n_3+1,n_4+1\}$ and so forth;
  \item if $j=3$, then $A_3^3=[n_3]$, $A_3^4=[n_3+2]$,
        $A_3^5=[n_3+2]\cup\{n_4+1,n_4+2\}$,
        $A_3^6=[n_3+2]\cup\{n_4+1,n_4+2,n_5+1,n_5+2\}$ and so forth.
 \end{enumerate}
\end{remark}

\begin{remark}
 Note that $A_j^m(\ve{n})$ depends only on $n_k$ for $k\in[\max\{j,m-1\}]$;
 in particular, $A_j^m(\ve{n})=A_j^m(\ve{n}')$ if $n_k=n_k'$ for all $k\in[m]$.
\end{remark}

\begin{proposition}\label{prop:A,S_basic}
 Let $\ve{n}\in Z$, $\ve{\delta}\in Y$, and $k\in\Z_+$.
 \begin{enumerate}
  \item $[n_j]=A_j^j(\ve{n})\subset A_j^{j+1}(\ve{n})
         \subset A_j^{j+2}(\ve{n})\subset\cdots$
        for every $j\in\N$, and
        $[n_1]=A_1^m(\ve{n})\subset\cdots\subset A_m^m(\ve{n})=[n_m]$
        for every $m\in\N$.
        In particular, $[n_j]\subset A_j^m(\ve{n})\subset[n_m]$
        whenever $j\le m$.
  \item $A_j^{m+1}(\ve{n}^k)\subset A_j^m(\ve{n}^{k+1})$
        for all $j,m\in\N$ with $j\le m$.
  \item $\Sr_k(\ve{n},\ve{\delta})=\Sr_0(\ve{n}^k,\ve{\delta})$.
  \item $\Sr_k(\ve{n},\ve{\delta})\subset\Sr_{k+1}(\ve{n},\ve{\delta})$.
 \end{enumerate}
\end{proposition}

\begin{proof}
 \begin{enumerate}
  \item Immediate from the definition.
  \item
  We have
  \begin{align*}
   A_j^m(\ve{n}^{k+1})
   &=[n_j^{k+1}]\cup\bigcup_{i=j}^{m-1}\{n_i^{k+1}+1,\dotsc,n_i^{k+1}+j-1\}\\
   &=[n_{j+1}^k]\cup\bigcup_{i=j}^{m-1}\{n_{i+1}^k+1,\dotsc,n_{i+1}^k+j-1\}\\
   &\supset[n_j^k+j-1]\cup\bigcup_{i=j+1}^{m}\{n_i^k+1,\dotsc,n_i^k+j-1\}\\
   &=[n_j^k]\cup\bigcup_{i=j}^{m}\{n_i^k+1,\dotsc,n_i^k+j-1\}\\
   &=A_j^{m+1}(\ve{n}^k).
  \end{align*}
  \item
  Immediate from the definition.
  \item
  Suppose that $\ve{K}\in\Sr_k(\ve{n},\ve{\delta})$ and $2\le j\le m-1$.
  Then we have
  \begin{align*}
   A_j^m(\ve{n}^{k+1})\setminus A_j^{m-1}(\ve{n}^{k+1})
   &=\{n_{m-1}^{k+1}+1,\dotsc,n_{m-1}^{k+1}+j-1\}\\
   &=\{n_m^k+1,\dotsc,n_m^k+j-1\}\\
   &=A_j^{m+1}(\ve{n}^k)\setminus A_j^m(\ve{n}^k),
  \end{align*}
  which, together with (2), implies that
  \begin{align*}
   \bigcup_{n\in A_j^m(\ve{n}^{k+1})\setminus A_j^{m-1}(\ve{n}^{k+1})}K_n
   &=\bigcup_{n\in A_j^{m+1}(\ve{n}^k)\setminus A_j^m(\ve{n}^k)}K_n\\
   &\subset\bigcup_{n\in A_{j-1}^m(\ve{n}^k)}B(K_n,\delta_{m+1})\\
   &\subset\bigcup_{n\in A_{j-1}^{m-1}(\ve{n}^{k+1})}B(K_n,\delta_m)
  \end{align*}
  because $\delta_{m+1}<\delta_m$.
  Hence we obtain $\ve{K}\in\Sr_{k+1}(\ve{n},\ve{\delta})$.\qedhere
 \end{enumerate}
\end{proof}

\begin{proposition}\label{prop:K_n_trick}
 Let $\ve{n}\in Z$, $\ve{\delta}\in Y$, and $k\in\Z_+$.
 If $\ve{K}\in\Sr_k(\ve{n},\ve{\delta})$, then
 \[
  \bigcap_{m=j}^{\infty}\bigcup_{n\in A_j^m(\ve{n}^k)}B(K_n,\delta_m)
  \subset\bigcup_{n=1}^{\infty}K_n
 \]
 for all $j\in\N$.
\end{proposition}

\begin{proof}
 By Proposition~\ref{prop:A,S_basic} (3), we may assume that $k=0$.
 For simplicity we write $A_j^m$ for $A_j^m(\ve{n})$.
 Fix $j\in\N$ and take any
 $x\in\bigcap_{m=j}^{\infty}\bigcup_{n\in A_j^m}B(K_n,\delta_m)$.
 Seeking a contradiction, suppose that $x\notin\bigcup_{n=1}^{\infty}K_n$.

 For each $i\in\N$, set $A_i=\bigcup_{m=i}^{\infty}A_i^m$
 and $L_i=\overline{\bigcup_{n\in A_i}K_n}$.
 Then we have
 \[
  x\in\bigcap_{m=j}^{\infty}\bigcup_{n\in A_j^m}B(K_n,\delta_m)
  \subset\bigcap_{m=j}^{\infty}B(L_j,\delta_m)
  =L_j,
 \]
 which allows us to define $i_0$ as the minimum $i\in\N$ with $x\in L_i$.

 If $i_0=1$, then $A_1=[n_1]$ and $x\in L_1=\bigcup_{n=1}^{n_1}K_n$,
 contradicting our assumption that $x\notin\bigcup_{n=1}^{\infty}K_n$.
 Thus $i_0\ge2$.

 For each $m\in\N$, take $x_m\in\bigcup_{n\in A_{i_0}}K_n$
 with $\lvert x_m-x\rvert<1/m$
 and choose $k_m\in A_{i_0}$ with $x_m\in K_{k_m}$.
 If there exists $k\in\N$ such that $k_m=k$ for infinitely many $m\in\N$,
 then $x=\lim_{m\to\infty}x_m\in K_k$,
 contradicting our assumption;
 therefore such $k$ does not exist.
 Consequently, for each $i\ge i_0$, we may take
 $m_i\in\N$ with $k_{m_i}\notin A_{i_0}^i$,
 and we may assume that $m_{i_0}<m_{i_0+1}<\cdots\to\infty$.
 Then for each $i\ge i_0$ we have
 \begin{align*}
  x_{m_i}
  &\in K_{k_{m_i}}
  \subset\bigcup_{n\in A_{i_0}\setminus A_{i_0}^i}K_n
  =\bigcup_{l=i+1}^{\infty}\bigcup_{n\in A_{i_0}^l\setminus A_{i_0}^{l-1}}K_n\\
  &\subset\bigcup_{l=i+1}^{\infty}\bigcup_{n\in A_{i_0-1}^{l-1}}B(K_n,\delta_l)
  \subset\bigcup_{l=i+1}^{\infty}\bigcup_{n\in A_{i_0-1}^{l-1}}B(K_n,\delta_{i+1})\\
  &\subset\bigcup_{n\in A_{i_0-1}}B(K_n,\delta_{i+1})
  \subset B(L_{i_0-1},\delta_{i+1}),
 \end{align*}
 keeping in mind that
 $\ve{K}\in\Sr_0(\ve{n},\ve{\delta})$ and $\ve{\delta}\in Y$.
 It follows that
 \[
  x\in\bigcap_{i=i_0}^{\infty}B(L_{i_0-1},\delta_{i+1}+1/m_i)
   =L_{i_0-1},
 \]
 which violates the minimality of $i_0$.
 This completes the proof.
\end{proof}

\begin{definition}\label{definition:Y,X}
 For $k\in\Z_+$, we define $\Y_k$ as the set of all
 \[
  (\ve{K},f,\ve{n},\ve{\delta},\ve{a},\ve{b})
  \in\K^{\N}\times C(I)\times Z\times Y\times X\times X
 \]
 such that $\ve{K}\in\Sr_k(\ve{n},\ve{\delta})$ and
 \[
  N(f,a_j)\subset\bigcup_{n\in A_j^m(\ve{n}^k)}B(K_n,\delta_m),\qquad
  \bigcup_{n\in A_j^m(\ve{n})}K_n\subset B\bigl(N(f,b_j),\delta_m\bigr)
 \]
 whenever $j\le m$.
 Set $\Y=\bigcup_{k=0}^{\infty}\Y_k$
 and write $\X$ for the projection of $\Y$ to $\K^{\N}\times C(I)$.
\end{definition}

\begin{remark}
 Note the difference between the subscripts of the two unions above.
\end{remark}

\begin{proposition}\label{prop:Y_increasing}
 We have $\Y_k\subset\Y_{k+1}$ for all $k\in\Z_+$.
\end{proposition}

\begin{proof}
 Thanks to Proposition~\ref{prop:A,S_basic} (4),
 it suffices to prove that
 \[
  \bigcup_{n\in A_j^{m+1}(\ve{n}^k)}B(K_n,\delta_{m+1})
  \subset\bigcup_{n\in A_j^m(\ve{n}^{k+1})}B(K_n,\delta_m)
 \]
 whenever $\ve{K}\in\K^{\N}$, $\ve{n}\in Z$, $\ve{\delta}\in Y$, and $j\le m$.
 Proposition~\ref{prop:A,S_basic} (2) shows that
 \begin{align*}
  \bigcup_{n\in A_j^m(\ve{n}^{k+1})}B(K_n,\delta_m)
  &\supset\bigcup_{n\in A_j^{m+1}(\ve{n}^k)}B(K_n,\delta_m)\\
  &\supset\bigcup_{n\in A_j^{m+1}(\ve{n}^k)}B(K_n,\delta_{m+1})
 \end{align*}
 because $\delta_m>\delta_{m+1}$.
\end{proof}

\begin{proposition}\label{prop:X_cup=N}
 If $(\ve{K},f)\in\X$, then $\bigcup_{n=1}^{\infty}K_n=N(f)$.
\end{proposition}

\begin{proof}
 Take $\ve{n}\in Z$, $\ve{\delta}\in Y$, $\ve{a},\ve{b}\in X$,
 and $k\in\Z_+$ so that
 $(\ve{K},f,\ve{n},\ve{\delta},\ve{a},\ve{b})\in\Y_k$.

 Firstly, we prove that $\bigcup_{n=1}^{\infty}K_n\subset N(f)$.
 Since
 \[
  \bigcup_{n=1}^{n_j}K_n
  =\bigcap_{m=j}^{\infty}\bigcup_{n\in A_j^m(\ve{n})}K_n
  \subset\bigcap_{m=j}^{\infty}B\bigl(N(f,b_j),\delta_m\bigr)
  =N(f,b_j)
 \]
 for every $j\in\N$, we have
 \[
  \bigcup_{n=1}^{\infty}K_n
  =\bigcup_{j=1}^{\infty}\bigcup_{n=1}^{n_j}K_n
  \subset\bigcup_{j=1}^{\infty}N(f,b_j)
  =N(f).
 \]

 Secondly, we prove that $N(f)\subset\bigcup_{n=1}^{\infty}K_n$.
 For every $j\in\N$, the definition of $\Y_k$ and Proposition~\ref{prop:K_n_trick}
 show that
 \[
  N(f,a_j)
  \subset\bigcap_{m=j}^{\infty}\bigcup_{n\in A_j^m(\ve{n}^k)}B(K_n,\delta_m)
  \subset\bigcup_{n=1}^{\infty}K_n.
 \]
 It follows that
 \[
  N(f)=\bigcup_{j=1}^{\infty}N(f,a_j)\subset\bigcup_{n=1}^{\infty}K_n.\qedhere
 \]
\end{proof}

\begin{lemma}\label{lem:perm_A}
 Let $\ve{n}\in Z$, and suppose that
 a permutation $\sigma$ on $\N$ and $k\in\N$ satisfy $\sigma(n)=n$ for all $n>n_k$.
 Then we have the following:
 \begin{enumerate}
  \item $A_j^m(\ve{n}^k)$ is invariant under $\sigma$ whenever $j\le m$;
  \item $\sigma\bigl(A_j^m(\ve{n})\bigr)\subset A_{\max\{j,k\}}^{\max\{m,k\}}(\ve{n})$
        whenever $j\le m$.
 \end{enumerate}
\end{lemma}

\begin{proof}
 Note that every subset of $\N$ that contains $[n_k]$ is invariant under $\sigma$.

 \begin{enumerate}
  \item The assertion follows from the observation that
        \[
         A_j^m(\ve{n}^k)\supset[n_j^k]=[n_{j+k}]\supset[n_k].
        \]
  \item If $k\le j$, then $A_j^m(\ve{n})\supset[n_j]\supset[n_k]$ and so
        $\sigma\bigl(A_j^m(\ve{n})\bigr)=A_j^m(\ve{n})$.
        If $j<k\le m$, then
        $\sigma\bigl(A_j^m(\ve{n})\bigr)
         \subset\sigma\bigl(A_k^m(\ve{n})\bigr)
         =A_k^m(\ve{n})$.
        If $m<k$, then
        $\sigma\bigl(A_j^m(\ve{n})\bigr)
         \subset\sigma([n_m])
         \subset\sigma([n_k])=[n_k]=A_k^k(\ve{n})$.\qedhere
 \end{enumerate}
\end{proof}

\begin{proposition}\label{prop:K's_invariant}
 If $f\in C(I)$, then $\{\ve{K}\in\K^{\N}\mid(\ve{K},f)\in\X\}$
 is invariant under finite permutations.
\end{proposition}

\begin{proof}
 Suppose that $\ve{K}$ belongs to the set and
 that $\sigma$ is a permutation on $\N$ that fixes all but finitely many
 positive integers.
 Define $\ve{K}'\in\K^{\N}$ by setting $K_n'=K_{\sigma(n)}$ for $n\in\N$.
 We need to prove that $(\ve{K}',f)\in\X$.

 Take $\ve{n}\in Z$, $\ve{\delta}\in Y$, $\ve{a},\ve{b}\in X$,
 and $k\in\Z_+$ so that $(\ve{K},f,\ve{n},\ve{\delta},\ve{a},\ve{b})\in\Y_k$.
 By Proposition~\ref{prop:Y_increasing},
 we may assume that $k$ is so large that $\sigma(n)=n$ for all $n>n_k$.

 By Lemma~\ref{lem:perm_A} (1),
 it is easy to see that $\ve{K}'\in\Sr_k(\ve{n},\ve{\delta})$ and
 that $N(f,a_j)\subset\bigcup_{n\in A_j^m(\ve{n}^k)}B(K_n',\delta_m)$ whenever $j\le m$.

 Now define $\ve{b}'\in X$ by setting $b_j'=b_{j+k}$ for $j\in\N$.
 Then for $j,m\in\N$ with $j\le m$,
 Lemma~\ref{lem:perm_A} (2) shows that
 \begin{align*}
  \bigcup_{n\in A_j^m(\ve{n})}K_n'
  &\subset\bigcup_{n\in A_{\max\{j,k\}}^{\max\{m,k\}}(\ve{n})}K_n
   \subset B\bigl(N(f,b_{\max\{j,k\}}),\delta_{\max\{m,k\}}\bigr)\\
  &\subset B\bigl(N(f,b_j'),\delta_m\bigr)
 \end{align*}
 because $b_{\max\{j,k\}}\le b_{j+k}=b_j'$ and
 $\delta_{\max\{m,k\}}\le\delta_m$.

 Hence we have shown that $(\ve{K}',f,\ve{n},\ve{\delta},\ve{a},\ve{b}')\in\Y_k$,
 from which it follows that $(\ve{K}',f)\in\X$.
\end{proof}

\begin{proposition}\label{prop:X_analytic}
 The set $\X$ is an analytic subset of $\K^{\N}\times C(I)$.
\end{proposition}

\begin{remark}
 For the following proof, tilde $\tilde{\ }$ does not have its usual meaning and
 is not related to hat $\hat{\ }$ or check $\check{\ }$ in the usual way.
\end{remark}

\begin{proof}[Proof of Proposition~\ref{prop:X_analytic}]
 Let
 $\pr\colon\K^{\N}\times C(I)\times Z\times Y\times X\times X
     \longrightarrow\K^{\N}\times C(I)$
 be the projection.
 It suffices to prove that $\pr\Y_k=\pr\bar{\Y}_k$ for every $k\in\Z_+$,
 because it will imply that
 \[
  \X=\pr\Y=\pr\Biggl(\bigcup_{k=0}^{\infty}\Y_k\Biggr)
  =\bigcup_{k=0}^{\infty}\pr\Y_k
  =\bigcup_{k=0}^{\infty}\pr\bar{\Y}_k,
 \]
 from which it follows that $\X$ is analytic.

 Let $k\in\Z_+$.
 We only need to prove that $\pr\bar{\Y}_k\subset\pr\Y_k$,
 so let $(\ve{K},f)\in\pr\bar{\Y}_k$ be given.
 Take $\ve{n}\in Z$, $\ve{\delta}\in Y$, $\ve{a},\ve{b}\in X$
 with $(\ve{K},f,\ve{n},\ve{\delta},\ve{a},\ve{b})\in\bar{\Y}_k$.
 Choosing $\ve{\delta}'\in Y$, $\ve{a}',\ve{b}'\in X$ so that
 $\delta_j'>\delta_j$, $a_j'<a_j$, $b_j'>b_j$ for all $j\in\N$,
 we shall show that $(\ve{K},f,\ve{n},\ve{\delta}',\ve{a}',\ve{b}')\in\Y_k$;
 it will imply that $(\ve{K},f)\in\pr\Y_k$, completing the proof.

 Firstly, we show that $\ve{K}\in\Sr_k(\ve{n},\ve{\delta}')$.
 Fix any $j_0,m_0\in\N$ with $2\le j_0\le m_0-1$.
 Take $\varepsilon>0$ with $\varepsilon<\delta_{m_0}'-\delta_{m_0}$.
 Since $(\ve{K},f,\ve{n},\ve{\delta},\ve{a},\ve{b})\in\bar{\Y}_k$,
 we may find
 $(\tilde{\ve{K}},\tilde{f},\tilde{\ve{n}},\tilde{\ve{\delta}},
   \tilde{\ve{a}},\tilde{\ve{b}})\in\Y_k$
 such that
 \begin{itemize}
  \item $d(\tilde{K}_n,K_n)<\varepsilon/2$ for $n\in[n_{m_0+k}]$;
  \item $\tilde{n}_j=n_j$ for $j\in[m_0+k]$;
  \item $\varepsilon<\delta_{m_0}'-\tilde{\delta}_{m_0}$.
 \end{itemize}
 We write $A_j^m=A_j^m(\ve{n}^k)$ and $\tilde{A}_j^m=A_j^m(\tilde{\ve{n}}^k)$
 for simplicity.
 Observe that
 \[
  \tilde{A}_{j_0}^{m_0}\setminus\tilde{A}_{j_0}^{m_0-1}
  =A_{j_0}^{m_0}\setminus A_{j_0}^{m_0-1},\qquad
  \tilde{A}_{j_0-1}^{m_0-1}=A_{j_0-1}^{m_0-1},
 \]
 and that if $n$ belongs to either of these sets,
 then $d(\tilde{K}_n,K_n)<\varepsilon/2$.
 Accordingly, we have
 \begin{align*}
  \bigcup_{n\in A_{j_0}^{m_0}\setminus A_{j_0}^{m_0-1}}K_n
  &\subset\bigcup_{n\in A_{j_0}^{m_0}\setminus A_{j_0}^{m_0-1}}
   B(\tilde{K}_n,\varepsilon/2)
   =\bigcup_{n\in\tilde{A}_{j_0}^{m_0}\setminus\tilde{A}_{j_0}^{m_0-1}}
   B(\tilde{K}_n,\varepsilon/2)\\
  &\subset\bigcup_{n\in\tilde{A}_{j_0-1}^{m_0-1}}
   B(\tilde{K}_n,\tilde{\delta}_{m_0}+\varepsilon/2)
   =\bigcup_{n\in A_{j_0-1}^{m_0-1}}
   B(\tilde{K}_n,\tilde{\delta}_{m_0}+\varepsilon/2)\\
  &\subset\bigcup_{n\in A_{j_0-1}^{m_0-1}}
   B(K_n,\tilde{\delta}_{m_0}+\varepsilon)
   \subset\bigcup_{n\in A_{j_0-1}^{m_0-1}}
   B(K_n,\delta_{m_0}').
 \end{align*}
 Hence we obtain $\ve{K}\in\Sr_k(\ve{n},\ve{\delta}')$.

 Now what remains to be shown is that if $j_0\le m_0$, then
 \[
  N(f,a_{j_0}')\subset\bigcup_{n\in A_{j_0}^{m_0}(\ve{n}^k)}B(K_n,\delta_{m_0}'),\qquad
  \bigcup_{n\in A_{j_0}^{m_0}(\ve{n})}K_n
  \subset B\bigl(N(f,b_{j_0}'),\delta_{m_0}'\bigr).
 \]
 Fix such $j_0$ and $m_0$, and take $\varepsilon>0$ with
 $\varepsilon<\delta_{m_0}'-\delta_{m_0}$.
 Since $(\ve{K},f,\ve{n},\ve{\delta},\ve{a},\ve{b})\in\bar{\Y}_k$,
 we may find
 $(\tilde{\ve{K}},\tilde{f},\tilde{\ve{n}},\tilde{\ve{\delta}},
   \tilde{\ve{a}},\tilde{\ve{b}})\in\Y_k$
 such that
 \begin{itemize}
  \item $d(\tilde{K}_n,K_n)<\varepsilon/2$ for $n\in[n_{m_0+k}]$;
  \item $\tilde{n}_j=n_j$ for $j\in[m_0+k]$;
  \item $\varepsilon<\delta_{m_0}'-\tilde{\delta}_{m_0}$;
  \item $a_{j_0}'<\tilde{a}_{j_0}$ and $b_{j_0}'>\tilde{b}_{j_0}$;
  \item $N(f,a_{j_0}')\subset B\bigl(N(\tilde{f},\tilde{a}_{j_0}),\varepsilon/2\bigr)$
        and
        $N(\tilde{f},\tilde{b}_{j_0})\subset B\bigl(N(f,b_{j_0}'),\varepsilon/2\bigr)$,
        which can be established because of Proposition~\ref{prop:continuity}.
 \end{itemize}
 Observe that
 \[
  A_{j_0}^{m_0}(\tilde{\ve{n}}^k)=A_{j_0}^{m_0}(\ve{n}^k),\qquad
  A_{j_0}^{m_0}(\tilde{\ve{n}})=A_{j_0}^{m_0}(\ve{n}),
 \]
 and that if $n$ belongs to either of these sets,
 then $d(\tilde{K}_n,K_n)<\varepsilon/2$.
 Accordingly, we have
 \begin{align*}
  N(f,a_{j_0}')
  &\subset B\bigl(N(\tilde{f},\tilde{a}_{j_0}),\varepsilon/2\bigr)
   \subset\bigcup_{n\in A_{j_0}^{m_0}(\tilde{\ve{n}}^k)}
   B(\tilde{K}_n,\tilde{\delta}_{m_0}+\varepsilon/2)\\
  &=\bigcup_{n\in A_{j_0}^{m_0}(\ve{n}^k)}
   B(\tilde{K}_n,\tilde{\delta}_{m_0}+\varepsilon/2)
   \subset\bigcup_{n\in A_{j_0}^{m_0}(\ve{n}^k)}
   B(K_n,\tilde{\delta}_{m_0}+\varepsilon)\\
  &\subset\bigcup_{n\in A_{j_0}^{m_0}(\ve{n}^k)}B(K_n,\delta_{m_0}')
 \end{align*}
 and
 \begin{align*}
  \bigcup_{n\in A_{j_0}^{m_0}(\ve{n})}K_n
  &\subset\bigcup_{n\in A_{j_0}^{m_0}(\ve{n})}B(\tilde{K}_n,\varepsilon/2)
   =\bigcup_{n\in A_{j_0}^{m_0}(\tilde{\ve{n}})}B(\tilde{K}_n,\varepsilon/2)\\
  &\subset B\bigl(N(\tilde{f},\tilde{b}_{j_0}),\tilde{\delta}_{m_0}+\varepsilon/2\bigr)
   \subset B\bigl(N(f,b_{j_0}'),\tilde{\delta}_{m_0}+\varepsilon\bigr)\\
  &\subset B\bigl(N(f,b_{j_0}'),\delta_{m_0}'\bigr).\qedhere
 \end{align*}
\end{proof}

\subsection{Key Proposition}
We reduce the main theorem
(Theorem~\ref{theorem:maintheorem_general} or equivalently Theorem~\ref{theorem:maintheorem2})
to a proposition, which we shall refer to as \emph{Key Proposition}.

\begin{proposition}[Key Proposition]\label{prop:key_prop}
 If $\Ar$ is a residual subset of $\K^{\N}$, then
 a typical function $f\in C(I)$ has the property that
 $(\ve{K},f)\in\X$ for some $\ve{K}\in\Ar$.
\end{proposition}

The proof of the key proposition will be given in the next section;
here we only show that it implies the main theorem.

\begin{proposition}
 The key proposition implies the main theorem.
 That is to say, if the key proposition is true,
 then a subfamily $\F$ of $\F_{\sigma}$ is residual if and only if
 $N(f)\in\F$ for a typical function $f\in C(I)$.
\end{proposition}

\begin{proof}
 Suppose first that $\F$ is residual.
 Then the key proposition applied to
 $\Ar=\{\ve{K}\in\K^{\N}\mid\bigcup_{n=1}^{\infty}K_n\in\F\}$
 tells us that a typical function $f\in C(I)$ has the property that
 $(\ve{K},f)\in\X$ for some $\ve{K}\in\Ar$,
 which implies that $N(f)=\bigcup_{n=1}^{\infty}K_n\in\F$
 by Proposition~\ref{prop:X_cup=N}.

 Conversely, suppose that a typical function $f\in C(I)$ has the property that
 $N(f)\in\F$.
 Then we may take a dense $G_{\delta}$ subset $G$ of $C(I)$
 contained in $\{f\in C(I)\mid N(f)\in\F\}$.
 Write $\Ar$ for the set of all $\ve{K}\in\K^{\N}$
 such that $(\ve{K},f)\in\X$ for some $f\in G$.
 Observe that $\Ar$ is invariant under finite permutations
 because it is a union of sets invariant under finite permutations
 by Proposition~\ref{prop:K's_invariant}.
 Since $\Ar$ is the projection of $\X\cap(\K^{\N}\times G)$ to $\K^{\N}$,
 Proposition~ \ref{prop:X_analytic} shows that $\Ar$ is analytic,
 and so $\Ar$ has the Baire property.
 Therefore Proposition~\ref{prop:zero-one} implies that
 $\Ar$ is either meagre or residual.
 If $\Ar$ is meagre, then the key proposition applied to $\Ar^c$ and
 the residuality of $G$ imply that $(\ve{K},f)\in\X$
 for some $f\in G$ and $\ve{K}\in\Ar^c$,
 which contradicts the definition of $\Ar$.
 Hence $\Ar$ is residual.
 This completes the proof because if $\ve{K}\in\Ar$, then
 for some $f\in G$ we have $\bigcup_{n=1}^{\infty}K_n=N(f)\in\F$
 by Proposition~\ref{prop:X_cup=N}.
\end{proof}

\section{Proof of the key proposition}
This section will be devoted to the proof of
the key proposition (Proposition~\ref{prop:key_prop}).
Let $\Ar$ be a residual subset of $\K^{\N}$,
and define
\[
 S=\{f\in C(I)\mid\text{$(\ve{K},f)\in\X$ for some $\ve{K}\in\Ar$}\}.
\]
We need to prove that $S$ is residual in $C(I)$.

\subsection{Banach-Mazur game}
We use the Banach-Mazur game to prove that $S$ is residual.

\begin{definition}[Banach-Mazur game]\label{defn:BMgame}
 The \emph{Banach-Mazur game} is described as follows.
 Two players, called Player I and Player II, alternately choose
 an open ball in $C(I)$ whose centre is a $C^1$ function,
 with the restriction that
 each player must choose a subset of the set
 chosen by the other player in the previous turn.
 Player~II will win if the intersection of all the sets chosen by the players
 is contained in $S$;
 otherwise Player~I will win.
\end{definition}

There is an easy criterion for deciding whether Player~II has
a winning strategy in the Banach-Mazur game:

\begin{theorem}[{{\cite[Theorem~1]{Oxtoby}}}]\label{thm:BMgame}
 The Banach-Mazur game admits a winning strategy for Player~II
 if and only if $S$ is residual in $C(I)$.
\end{theorem}

Therefore it suffices to construct a winning strategy
for Player~II in the Banach-Mazur game.

\subsection{Introduction to the strategy}
Since $\Ar$ is residual,
we may take open dense subsets $\U_m$ of $\K^{\N}$ for $m\in\N$
so that $\bigcap_{m=1}^{\infty}\U_m\subset\Ar$.

We shall use two sequences of positive numbers $a_j$ and $b_j$,
and their cousins $a_j^{m,k}$ and $b_j^{m,k}$.
The numbers $a_j$ are defined by $a_j=j$ for $j\in\N$,
and the numbers $a_j^{m,k}$, where $j\le m$ and $k\in[4]$, are chosen
to satisfy
\begin{align*}
 a_{j+1}=j+1
 &>a_j^{j,1}>a_j^{j,2}>a_j^{j,3}>a_j^{j,4}\\
 &=a_j^{j+1,1}>a_j^{j+1,2}>a_j^{j+1,3}>a_j^{j+1,4}\\
 &=\cdots\\
 &\to a_j=j
\end{align*}
(for example, $a_j^{m,k}=j+2^{-(3m+k)}$).
The numbers $b_j$ are defined in the strategy,
each $b_j$ being determined in the $j$th round,
and they satisfy $b_j<b_{j+1}$ and $b_j>j+2$ for all $j\in\N$.
As soon as each $b_j$ is determined,
the numbers $b_j^{m,k}$ for $m\ge j$ and $k\in[3]$ are chosen to satisfy
\begin{align*}
 j+1<b_j-1
    &<b_j^{j,1}<b_j^{j,2}<b_j^{j,3}\\
    &=b_j^{j+1,1}<b_j^{j+1,2}<b_j^{j+1,3}\\
    &=\cdots\\
    &\to b_j
\end{align*}
(for example, $b_j^{m,k}=b_j-2^{-(2m+k)}$).
Note that $a_j^{m,k}<j+1<b_j^{m',k'}$ for all $j$, $m$, $m'$, $k$, $k'$.

The moves of Players~I and II in the $m$th round will be denoted
by $B(f_m,\alpha_m)$ and $B(g_m,\beta_m)$ respectively.
By the rule of the game, the functions $f_m$ and $g_m$ are
all continuously differentiable.
In the $m$th round, Player~II will construct,
in addition to $g_m$ and $\beta_m$, the following:
a positive number $h_m$,
a positive number $\mu_m$,
finite subsets $\tilde{L}_n^m$ of $I$,
a sequence $\ve{K}^m\in\K^{\N}$
(and its partition $K_n^m=\hat{K}_n^m\amalg\check{K}_n^m$),
a positive integer $n_m$,
a positive number $w_m$,
and a positive number $b_m$ (as mentioned above).
They will be chosen to satisfy a number of properties,
but the following, written as $(\bigstar_m)$ afterwards,
is essential to ensure that the induction proceeds:
if $f\in B(g_m,\beta_m)$, then
\begin{itemize}
 \item $\tilde{N}(f,a_j^{m,4})\subset\bigcup_{n\in A_j^m}B(\tilde{K}_n^m,w_m)$,
 \item $\bigcup_{n\in A_j^m}\tilde{K}_n^m
        \subset B\bigl(\tilde{N}(f,b_j^{m,3}),w_m\bigr)$,
 \item $\tilde{N}(f,a_j^{m,4})
        \cap\bigcup_{n\in[n_m]\setminus A_j^m}\bar{B}(\tilde{K}_n^m,w_m)
        =\emptyset$
\end{itemize}
for every $j\in[m]$.
Here $A_j^m=A_j^m(\ve{n})$, where $\ve{n}=(n_m)$ is the sequence of positive integers
whose $m$th term will be defined in the $m$th round by Player~II.
We must be careful exactly when $A_j^m$ will be determined;
it is true that the whole sequence $\ve{n}$ will be determined only
after the game is over, but since $A_j^m$ depends only on
$n_k$ for $k\in[\max\{j,m-1\}]$,
we can use $A_j^m$ once $n_{\max\{j,m-1\}}$ is determined.

\subsection{First round}
Suppose that Player~I has given his first move $B(f_1,\alpha_1)$.

Let $\D$ denote the dense subset of $\K^{\N}$ consisting of all sequences
whose terms are pairwise disjoint finite sets.
For $\ve{M}\in\K^{\N}$, $l\in\N$, and $r>0$, we set
\[
 \bar{U}(\ve{M},l,r)
 =\{\ve{M}'\in\K^{\N}\mid\text{$d(M_n,M_n')\le r$ for all $n\in[l]$}\}.
\]

\subsubsection{Construction of $h_1$, $\mu_1$, $L_n^1$, $\ve{K}^1$, $n_1$, and $w_1$}
Take $h_1>0$ with $h_1<\alpha_1$,
and set $\mu_1=\mu(f_1,a_1^{1,3},a_1^{1,2},h_1)$
(recall Definition~\ref{definition:mu}).
Put $\tilde{L}_n^1=\emptyset$ for every $n\in\N$.
There exists $\ve{K}^1\in\U_1\cap\D$ such that
we may partition $K_1^1$ as $K_1^1=\hat{K}_1^1\amalg\check{K}_1^1$
in such a way that $B(\tilde{K}_1^1,\mu_1)=I$.
Choose $n_1\in\N$ and $w_1>0$ so that $\bar{U}(\ve{K}^1,n_1,2w_1)\subset\U_1$;
make $w_1$ smaller, if necessary, so that
the balls $\bar{B}(x,w_1)$ for $x\in\bigcup_{n=1}^{n_1}K_n^1$
are disjoint.

\subsubsection{Construction of $g_1$ and $b_1$}
Let $\varphi_1$ be a bump function of height $h_1$ and width $w_1$
located at $\hat{K}_1^1$ and $\check{K}_1^1$.
Define $g_1=f_1+\varphi_1$.
It is clear that $g_1\in B(f_1,\alpha_1)$.
Since $\mu_1=\mu(f_1,a_1^{1,3},a_1^{1,2},h_1)$ and $B(\tilde{K}_1^1,\mu_1)=I$,
Proposition~\ref{prop:bump_difficult} shows that
\[
 \tilde{N}(g_1,a_1^{1,3})
 \subset\tilde{N}(f_1,a_1^{1,2})\cap B(\tilde{K}_1^1,w_1)
 \subset B(\tilde{K}_1^1,w_1)
 \subset\bigcup_{n=1}^{n_1}B(\tilde{K}_n^1,w_1).
\]

Let $b_1>3$ be so large that $b_1^{1,2}\ge\lVert g_1'\rVert$.
Then $\tilde{N}(g_1,b_1^{1,2})=I\supset\bigcup_{n=1}^{n_1}\tilde{K}_n^1$.

Since $A_1^1=[n_1]$, we have
\begin{itemize}
 \item $\tilde{N}(g_1,a_1^{1,3})\subset\bigcup_{n\in A_1^1}B(\tilde{K}_n^1,w_1)$;
 \item $\bigcup_{n\in A_1^1}\tilde{K}_n^1\subset\tilde{N}(g_1,b_1^{1,2})$.
\end{itemize}

\subsubsection{Construction of $\beta_1$}
We may find $\varepsilon_1>0$ such that
\begin{itemize}
 \item $\tilde{N}(g_1,a_1^{1,3})
        \subset\bigcup_{n\in A_1^1}B(\tilde{K}_n^1,w_1-\varepsilon_1)$.
\end{itemize}
By Proposition~\ref{prop:continuity},
there exists $\beta_1>0$ with $B(g_1,\beta_1)\subset B(f_1,\alpha_1)$
such that whenever $f\in B(g_1,\beta_1)$, we have
\begin{itemize}
 \item $\tilde{N}(f,a_1^{1,4})
        \subset B\bigl(\tilde{N}(g_1,a_1^{1,3}),\varepsilon_1\bigr)$;
 \item $\tilde{N}(g_1,b_1^{1,2})
        \subset B\bigl(\tilde{N}(f,b_1^{1,3}),w_1\bigr)$.
\end{itemize}
It follows that whenever $f\in B(g_1,\beta_1)$, we have
\begin{itemize}
 \item $\tilde{N}(f,a_1^{1,4})
        \subset\bigcup_{n\in A_1^1}B(\tilde{K}_n^1,w_1)$;
 \item $\bigcup_{n\in A_1^1}\tilde{K}_n^1
        \subset B\bigl(\tilde{N}(f,b_1^{1,3}),w_1\bigr)$;
 \item $\tilde{N}(f,a_1^{1,4})
        \cap\bigcup_{n\in[n_1]\setminus A_1^1}\bar{B}(\tilde{K}_n^1,w_1)=\emptyset$,
\end{itemize}
the last condition being trivial because $[n_1]\setminus A_1^1=\emptyset$.
Therefore $(\bigstar_1)$ has been established.

\subsection{$m$th round for $m\ge2$}
Let $m\ge2$ and suppose that Player~I has given his $m$th move $B(f_m,\alpha_m)$.
Since the rule of the Banach-Mazur game requires that $f_m\in B(g_{m-1},\beta_{m-1})$,
it follows from $(\bigstar_{m-1})$ that
\begin{itemize}
 \item $\tilde{N}(f_m,a_j^{m,1})
        \subset\bigcup_{n\in A_j^{m-1}}B(\tilde{K}_n^{m-1},w_{m-1})$,
 \item $\bigcup_{n\in A_j^{m-1}}\tilde{K}_n^{m-1}
        \subset B\bigl(\tilde{N}(f_m,b_j^{m,1}),w_{m-1}\bigr)$,
 \item $\tilde{N}(f_m,a_j^{m,1})
        \cap\bigcup_{n\in[n_{m-1}]\setminus A_j^{m-1}}\bar{B}(\tilde{K}_n^{m-1},w_{m-1})
        =\emptyset$
\end{itemize}
for every $j\in[m-1]$
(remember that $a_j^{m-1,4}=a_j^{m,1}$ and $b_j^{m-1,3}=b_j^{m,1}$).

\subsubsection{Construction of $h_m$ and $\mu_m$}
Take $h_m>0$ with $h_m<\alpha_m$, and set
\[
 \mu_m=\min_{j\in[m]}\mu(f_m,a_j^{m,3},a_j^{m,2},h_m)>0.
\]

\subsubsection{Construction of $L_n^m$}
Choosing an auxiliary number $\zeta_m>0$ so that
\begin{itemize}
 \item $\tilde{N}(f_m,a_j^{m,1})
        \subset\bigcup_{n\in A_j^{m-1}}B(\tilde{K}_n^{m-1},w_{m-1}-\zeta_m)$
       for $j\in[m-1]$,
\end{itemize}
we shall define finite subsets $\tilde{L}_n^m$ of $I$ for $n\in[n_{m-1}+m-1]$.

Firstly, let $n\in[n_{m-1}]$ and take the minimum $j\in[m-1]$ with $n\in A_j^{m-1}$.
When $x$ varies in $\tilde{N}(f_m,b_j^{m,1})\cap B(\tilde{K}_n^{m-1},w_{m-1})$,
\begin{itemize}
 \item the open balls $B(x,w_{m-1})$ cover $\tilde{K}_n^{m-1}$;
 \item the open balls $B(x,\mu_m)$ cover
       $\tilde{N}(f_m,a_j^{m,1})\cap\bar{B}(\tilde{K}_n^{m-1},w_{m-1}-\zeta_m)$.
\end{itemize}
The compactness of the sets covered gives us a finite subset $\tilde{L}_n^m$
of $\tilde{N}(f_m,b_j^{m,1})\cap B(\tilde{K}_n^{m-1},w_{m-1})$ such that
\begin{itemize}
 \item $B(\tilde{L}_n^m,w_{m-1})\supset\tilde{K}_n^{m-1}$;
 \item $B(\tilde{L}_n^m,\mu_m)\supset
        \tilde{N}(f_m,a_j^{m,1})\cap\bar{B}(\tilde{K}_n^{m-1},w_{m-1}-\zeta_m)$.
\end{itemize}

Secondly, for $j\in[m-1]\setminus\{1\}$, we set
\[
 \tilde{P}_j^m
 =\bigl(\tilde{N}(f_m,a_j^{m,1})\setminus\Int \tilde{N}(f_m,a_{j-1}^{m,1})\bigr)
  \cap\bigcup_{n\in A_{j-1}^{m-1}}\bar{B}(\tilde{K}_n^{m-1},w_{m-1}-\zeta_m),
\]
and define $\tilde{L}_{n_{m-1}+j-1}^m$ as a finite subset
of $\tilde{P}_j^m$ such that $B(\tilde{L}_{n_{m-1}+j-1}^m,\mu_m)\supset\tilde{P}_j^m$.
This defines $\tilde{L}_n^m$ for $n\in[n_{m-1}+m-2]\setminus[n_{m-1}]$.

Lastly, we define $\tilde{L}_{n_{m-1}+m-1}^m$
as a finite subset of $I\setminus\bigcup_{n=1}^{n_{m-1}+m-2}B(\tilde{L}_n^m,\mu_m)$
such that
$B(\tilde{L}_{n_{m-1}+m-1}^m,\mu_m)
 \supset I\setminus\bigcup_{n=1}^{n_{m-1}+m-2}B(\tilde{L}_n^m,\mu_m)$.

Having defined $\tilde{L}_n^m$ for $n\in[n_{m-1}+m-1]$,
we prove the following claim.
Remember that since $n_1$, \dots, $n_{m-1}$ have already been defined,
we know $A_j^m$ for $j\in[m-1]$.

\begin{claim}\label{claim:L^m}
 We have the following:
 \begin{enumerate}
  \item $d(\tilde{L}_n^m,\tilde{K}_n^{m-1})<w_{m-1}$ for $n\in[n_{m-1}]$;
  \item $\tilde{N}(f_m,a_{m-1}^{m,1})
         \subset\bigcup_{n=1}^{n_{m-1}+m-2}B(\tilde{L}_n^m,\mu_m)$;
  \item $\bigcup_{n\in A_j^m}\tilde{L}_n^m\subset\tilde{N}(f_m,b_j^{m,1})$
        for $j\in[m-1]$;
  \item $\bigcup_{n=1}^{n_{m-1}+m-1}B(\tilde{L}_n^m,\mu_m)=I$;
  \item $\bigcup_{n\in A_j^m\setminus A_j^{m-1}}\tilde{L}_n^m
         \subset\bigcup_{n\in A_{j-1}^{m-1}}B(\tilde{L}_n^m,2w_{m-1})$
        for $j\in[m-1]\setminus\{1\}$;
  \item $\tilde{N}(f_m,a_j^{m,2})\cap
         \bigcup_{n\in[n_{m-1}+m-1]\setminus A_j^m}\tilde{L}_n^m=\emptyset$
        for $j\in[m-1]$.
 \end{enumerate}
\end{claim}

\begin{proof}
 \begin{enumerate}
  \item Both $\tilde{L}_n^m\subset B(\tilde{K}_n^{m-1},w_{m-1})$
        and $\tilde{K}_n^{m-1}\subset B(\tilde{L}_n^m,w_{m-1})$
        are clear from the definition of $\tilde{L}_n^m$.
  \item Let $x\in\tilde{N}(f_m,a_{m-1}^{m,1})$ and
        look at the minimum $j\in[m-1]$ with $x\in\tilde{N}(f_m,a_j^{m,1})$.

        If $j=1$, then the definition of $\zeta_m$ tells us that
        $x\in B(\tilde{K}_n^{m-1},w_{m-1}-\zeta_m)$ for some $n\in A_1^{m-1}=[n_1]$;
        for this $n$, the number $j$ taken in the definition of $\tilde{L}_n^m$
        must be $1$, so
        \[
         x\in\tilde{N}(f_m,a_1^{m,1})\cap B(\tilde{K}_n^{m-1},w_{m-1}-\zeta_m)
          \subset B(\tilde{L}_n^m,\mu_m).
        \]

        Now, suppose that $j\in[m-1]\setminus\{1\}$.
        Since $x\in\tilde{N}(f_m,a_j^{m,1})$, we may take $n\in A_j^{m-1}$
        with $x\in B(\tilde{K}_n^{m-1},w_{m-1}-\zeta_m)$.
        If $n\notin A_{j-1}^{m-1}$,
        then the number $j$ taken in the definition of $\tilde{L}_n^m$
        must be the same as our $j$, and so
        \[
         x\in\tilde{N}(f_m,a_j^{m,1})\cap B(\tilde{K}_n^{m-1},w_{m-1}-\zeta_m)
          \subset B(\tilde{L}_n^m,\mu_m).
        \]
        If $n\in A_{j-1}^{m-1}$, then $x\in\tilde{P}_j^m$ because
        $x\notin\tilde{N}(f_m,a_{j-1}^{m,1})
         \supset \Int\tilde{N}(f_m,a_{j-1}^{m,1})$
        by the minimality of $j$;
        therefore $x\in B(\tilde{L}_{n_{m-1}+j-1}^m,\mu_m)$,
        which implies the required inclusion.
  \item Let $x\in\bigcup_{n\in A_j^m}\tilde{L}_n^m$ and
        take $n\in A_j^m$ with $x\in\tilde{L}_n^m$.
        If $n\in A_j^{m-1}$,
        then taking the minimum $i$ with $n\in A_i^{m-1}$, we have
        \[
         x\in\tilde{L}_n^m\subset\tilde{N}(f_m,b_i^{m,1})
         \subset\tilde{N}(f_m,b_j^{m,1}).
        \]
        If $n\notin A_j^{m-1}$, then $j\ge2$ and $n_{m-1}+1\le n\le n_{m-1}+j-1$,
        from which it follows that
        \begin{align*}
         x&\in\tilde{L}_n^m\subset\tilde{P}_{n-n_{m-1}+1}^m
           \subset\tilde{N}(f_m,a_{n-n_{m-1}+1}^{m,1})\\
          &\subset\tilde{N}(f_m,a_j^{m,1})
           \subset\tilde{N}(f_m,b_j^{m,1}).
        \end{align*}
  \item Immediate from the definition of $\tilde{L}_{n_{m-1}+m-1}$.
  \item We have
        \begin{align*}
         \bigcup_{n\in A_j^m\setminus A_j^{m-1}}\tilde{L}_n^m
         &=\bigcup_{n=n_{m-1}+1}^{n_{m-1}+j-1}\tilde{L}_n^m
          \subset\bigcup_{k=2}^{j}\tilde{P}_k^m\\
         &\subset\bigcup_{k=2}^{j}\bigcup_{n\in A_{k-1}^{m-1}}
                 \bar{B}(\tilde{K}_n^{m-1},w_{m-1}-\zeta_m)\\
         &=\bigcup_{n\in A_{j-1}^{m-1}}\bar{B}(\tilde{K}_n^{m-1},w_{m-1}-\zeta_m)\\
         &\subset\bigcup_{n\in A_{j-1}^{m-1}}B(\tilde{L}_n^m,2w_{m-1}),
        \end{align*}
        where the last inclusion follows from (1).
  \item We need to show that
        $\tilde{N}(f_m,a_j^{m,2})\cap\tilde{L}_n^m=\emptyset$
        for $n\in[n_{m-1}+m-1]\setminus A_j^m$.
        There are three cases:
        $n\in[n_{m-1}]\setminus A_j^{m-1}$,
        $n_{m-1}+j\le n\le n_{m-1}+m-2$, and $n=n_{m-1}+m-1$.

        If $n\in[n_{m-1}]\setminus A_j^{m-1}$, then
        \[
         \tilde{N}(f_m,a_j^{m,2})\cap\tilde{L}_n^m
         \subset\tilde{N}(f_m,a_j^{m,1})\cap B(\tilde{K}_n^{m-1},w_{m-1})
         =\emptyset
        \]
        by $(\bigstar_{m-1})$.

        If $n_{m-1}+j\le n\le n_{m-1}+m-2$, then
        \begin{align*}
         \tilde{N}(f_m,a_j^{m,2})\cap\tilde{L}_n^m
         &\subset\tilde{N}(f_m,a_j^{m,2})\cap\tilde{P}_{n-n_{m-1}+1}^m\\
         &\subset\tilde{N}(f_m,a_{n-n_{m-1}}^{m,2})
                 \setminus\Int\tilde{N}(f_m,a_{n-n_{m-1}}^{m,1})\\
         &=\emptyset
        \end{align*}
        because of Corollary~\ref{cor:C1_continuity_single_interior}.

        If $n=n_{m-1}+m-1$, then (2) implies that
        \begin{align*}
         \tilde{N}(f_m,a_j^{m,2})\cap\tilde{L}_n^m
         &\subset\tilde{N}(f_m,a_{m-1}^{m,1})\cap\tilde{L}_n^m\\
         &\subset\bigcup_{n'=1}^{n_{m-1}+m-2}B(\tilde{L}_{n'}^m,\mu_m)
           \cap\tilde{L}_{n_{m-1}+m-1}^m\\
         &=\emptyset
        \end{align*}
        because of the choice of $\tilde{L}_{n_{m-1}+m-1}^m$.\qedhere
 \end{enumerate}
\end{proof}

\subsubsection{Construction of $\ve{K}^m$}
We shall construct a sequence $\ve{K}^m\in\U_m\cap\D$
such that we may partition $K_n^m=\hat{K}_n^m\amalg\check{K}_n^m$
for each $n\in\N$ in such a way that the following conditions are fulfilled:
\begin{enumerate}
 \item $d(\tilde{K}_n^m,\tilde{K}_n^{m-1})<w_{m-1}$ for $n\in[n_{m-1}]$;
 \stepcounter{enumi}
 \item $\bigcup_{n\in A_j^m}\tilde{K}_n^m\subset\Int\tilde{N}(f_m,b_j^{m,2})$
       for $j\in[m-1]$;
 \item $\bigcup_{n=1}^{n_{m-1}+m-1}B(\tilde{K}_n^m,\mu_m)=I$;
 \item $\bigcup_{n\in A_j^m\setminus A_j^{m-1}}\tilde{K}_n^m
        \subset\bigcup_{n\in A_{j-1}^{m-1}}B(\tilde{K}_n^m,2w_{m-1})$
       for $j\in[m-1]\setminus\{1\}$;
 \item $\tilde{N}(f_m,a_j^{m,2})\cap
        \bigcup_{n\in[n_{m-1}+m-1]\setminus A_j^m}\tilde{K}_n^m=\emptyset$
       for $j\in[m-1]$
\end{enumerate}
(these are the relations of Claim~\ref{claim:L^m} (1), (3), (4), (5), (6) with
$\tilde{L}_n^m$ replaced by $\tilde{K}_n^m$ and
with $\tilde{N}(f_m,b_j^{m,1})$ replaced by $\Int\tilde{N}(f_m,b_j^{m,2})$ in (3)).

We note that Claim~\ref{claim:L^m} (3)
and Corollary~\ref{cor:C1_continuity_single_interior}
show that
$\bigcup_{n\in A_j^m}\tilde{L}_n^m\subset\Int\tilde{N}(f_m,b_j^{m,2})$
for $j\in[m-1]$.
Therefore, by Claim~\ref{claim:L^m},
if we choose \emph{disjoint} finite subsets
$\hat{Q}_1^m$, \dots, $\hat{Q}_{n_{m-1}+m-1}^m$,
$\check{Q}_1^m$, \dots, $\check{Q}_{n_{m-1}+m-1}^m$ of $I$
so that the distances $d(\tilde{Q}_n^m,\tilde{L}_n^m)$
for $n\in[n_{m-1}+m-1]$ are sufficiently small,
then they satisfy the following conditions:
\begin{enumerate}
 \item $d(\tilde{Q}_n^m,\tilde{K}_n^{m-1})<w_{m-1}$ for $n\in[n_{m-1}]$;
 \stepcounter{enumi}
 \item $\bigcup_{n\in A_j^m}\tilde{Q}_n^m\subset\Int\tilde{N}(f_m,b_j^{m,2})$
       for $j\in[m-1]$;
 \item $\bigcup_{n=1}^{n_{m-1}+m-1}B(\tilde{Q}_n^m,\mu_m)=I$;
 \item $\bigcup_{n\in A_j^m\setminus A_j^{m-1}}\tilde{Q}_n^m
        \subset\bigcup_{n\in A_{j-1}^{m-1}}B(\tilde{Q}_n^m,2w_{m-1})$
       for $j\in[m-1]\setminus\{1\}$;
 \item $\tilde{N}(f_m,a_j^{m,2})\cap
        \bigcup_{n\in[n_{m-1}+m-1]\setminus A_j^m}\tilde{Q}_n^m=\emptyset$
       for $j\in[m-1]$.
\end{enumerate}
Since $\ve{K}^m$ must belong to $\U_m$,
we consider $\ve{K}^m\in\U_m\cap\D$ such that
the distances $d(K_n^m,\hat{Q}_n^m\amalg\check{Q}_n^m)$ for $n\in[n_{m-1}+m-1]$
are so small that each point in $K_n^m$ has the unique closest point
in $\hat{Q}_n^m\amalg\check{Q}_n^m$.
If the distances $d(K_n^m,\hat{Q}_n^m\amalg\check{Q}_n^m)$ are sufficiently small,
the sequence $\ve{K}^m$ satisfies the required conditions.

\subsubsection{Construction of $n_m$ and $w_m$}
Choose $n_m\in\N$ and $w_m>0$ so that
\begin{itemize}
 \item $n_m\ge n_{m-1}+m-1$;
 \item $w_m<w_{m-1}/2$;
 \item $\bar{U}(\ve{K}^m,n_m,2w_m)\subset\U_m$;
 \item $\tilde{N}(f_m,a_j^{m,2})
        \cap\bigcup_{n\in[n_{m-1}+m-1]\setminus A_j^m}\bar{B}(\tilde{K}_n^m,w_m)
        =\emptyset$
       for $j\in[m-1]$.
\end{itemize}
Make $w_m$ smaller, if necessary, so that
\begin{itemize}
 \item the balls $\bar{B}(x,w_m)$ for $x\in\bigcup_{n=1}^{n_m}K_n^m$ are disjoint.
\end{itemize}

\subsubsection{Construction of $g_m$ and $b_m$}
Take a bump function $\varphi_m$ of height $h_m$ and width $w_m$
located at $\bigcup_{n=1}^{n_{m-1}+m-1}\hat{K}_n^m$
and $\bigcup_{n=1}^{n_{m-1}+m-1}\check{K}_n^m$,
and set $g_m=f_m+\varphi_m$.

Let $b_m>\max\{m+2,b_{m-1}\}$ be so large that $b_m^{m,2}\ge\lVert g_m'\rVert$.

\begin{claim}\label{claim:g_m}
 \begin{enumerate}
  \item $\tilde{N}(g_m,a_j^{m,3})\subset\bigcup_{n\in A_j^m}B(\tilde{K}_n^m,w_m)$
        for $j\in[m]$.
  \item $\bigcup_{n\in A_j^m}\tilde{K}_n^m\subset\tilde{N}(g_m,b_j^{m,2})$
        for $j\in[m]$.
  \item $\tilde{N}(g_m,a_j^{m,3})
         \cap\bigcup_{n\in[n_m]\setminus A_j^m}\bar{B}(\tilde{K}_n^m,w_m)=\emptyset$
        for $j\in[m]$.
 \end{enumerate}
\end{claim}

\begin{proof}
 \begin{enumerate}
  \item Remember the definition of $\mu_m$ and property (4) of $\ve{K}^m$.
        If $j=m$, then $A_j^m=A_m^m=[n_m]$ and
        \begin{align*}
         \tilde{N}(g_m,a_m^{m,3})
         &\subset\tilde{N}(f_m,a_m^{m,2})
                 \cap\bigcup_{n=1}^{n_{m-1}+m-1}B(\tilde{K}_n^m,w_m)\\
         &\subset\bigcup_{n=1}^{n_{m-1}+m-1}B(\tilde{K}_n^m,w_m)
          \subset\bigcup_{n=1}^{n_m}B(\tilde{K}_n^m,w_m).
        \end{align*}
        If $j\in[m-1]$, then the choice of $w_m$ implies that
        \begin{align*}
         \tilde{N}(g_m,a_j^{m,3})
         &\subset\tilde{N}(f_m,a_j^{m,2})
                 \cap\bigcup_{n=1}^{n_{m-1}+m-1}B(\tilde{K}_n^m,w_m)\\
         &\subset\bigcup_{n\in A_j^m}B(\tilde{K}_n^m,w_m).
        \end{align*}
  \item If $j=m$, then the choice of $b_m$ implies that
        \[
         \tilde{N}(g_m,b_j^{m,2})=\tilde{N}(g_m,b_m^{m,2})=I
         \supset\bigcup_{n\in A_j^m}\tilde{K}_n^m.
        \]
        If $j\in[m-1]$, then property (3) of $\ve{K}^m$ and
        Proposition~\ref{prop:bump_easy} show that
        \[
         \bigcup_{n\in A_j^m}\tilde{K}_n^m
         \subset\bigcup_{n=1}^{n_{m-1}+m-1}\tilde{K}_n^m\cap\tilde{N}(f_m,b_j^{m,2})
         \subset\tilde{N}(g_m,b_j^{m,2}).
        \]
  \item If $j=m$, then the claim is trivial because $[n_m]\setminus A_j^m=\emptyset$.
        If $j\in[m-1]$, then (1) and the choice of $w_m$ show that
        \begin{align*}
         &\tilde{N}(g_m,a_j^{m,3})
          \cap\bigcup_{n\in[n_m]\setminus A_j^m}\bar{B}(\tilde{K}_n^m,w_m)\\
         &\qquad\subset
          \bigcup_{n\in A_j^m}B(\tilde{K}_n^m,w_m)
          \cap\bigcup_{n\in[n_m]\setminus A_j^m}\bar{B}(\tilde{K}_n^m,w_m)
          =\emptyset.\qedhere
        \end{align*}
 \end{enumerate}
\end{proof}

\subsubsection{Construction of $\beta_m$}
We choose $\beta_m>0$ as in the following claim:

\begin{claim}\label{claim:beta_m}
 There exists $\beta_m>0$ with $B(g_m,\beta_m)\subset B(f_m,\alpha_m)$ such that
 if $f\in B(g_m,\beta_m)$, then
 \begin{itemize}
  \item $\tilde{N}(f,a_j^{m,4})\subset\bigcup_{n\in A_j^m}B(\tilde{K}_n^m,w_m)$,
  \item $\bigcup_{n\in A_j^m}\tilde{K}_n^m\subset
         B\bigl(\tilde{N}(f,b_j^{m,3}),w_m\bigr)$,
  \item $\tilde{N}(f,a_j^{m,4})
         \cap\bigcup_{n\in[n_m]\setminus A_j^m}\bar{B}(\tilde{K}_n^m,w_m)=\emptyset$
 \end{itemize}
 for every $j\in[m]$.
\end{claim}

\begin{proof}
 By Claim~\ref{claim:g_m}, we may find $\varepsilon_m>0$ such that
 \begin{itemize}
  \item $\tilde{N}(g_m,a_j^{m,3})
         \subset\bigcup_{n\in A_j^m}B(\tilde{K}_n^m,w_m-\varepsilon_m)$,
  \item $\bigcup_{n\in A_j^m}\tilde{K}_n^m\subset\tilde{N}(g_m,b_j^{m,2})$,
  \item $B\bigl(\tilde{N}(g_m,a_j^{m,3}),\varepsilon_m\bigr)
         \cap\bigcup_{n\in[n_m]\setminus A_j^m}\bar{B}(\tilde{K}_n^m,w_m)=\emptyset$
 \end{itemize}
 for every $j\in[m]$
 (note that there is no $\varepsilon_m$ in the second condition).
 By Proposition~\ref{prop:continuity},
 there exists $\beta_m>0$ with $B(g_m,\beta_m)\subset B(f_m,\alpha_m)$
 such that if $f\in B(g_m,\beta_m)$, then
 \begin{itemize}
  \item $\tilde{N}(f,a_j^{m,4})
         \subset B\bigl(\tilde{N}(g_m,a_j^{m,3}),\varepsilon_m\bigr)$,
  \item $\tilde{N}(g_m,b_j^{m,2})
         \subset B\bigl(\tilde{N}(f,b_j^{m,3}),w_m\bigr)$
 \end{itemize}
 for every $j\in[m]$.
 It is easy to see that this $\beta_m$ satisfies the required condition.
\end{proof}

\subsection{Proof that the strategy makes Player~II win}
\begin{proposition}\label{prop:K_n^m_convergence}
 \begin{enumerate}
  \item For every $n\in\N$, the sequence $(K_n^m)_{m\in\N}$ converges in $\K$.
        Denote the limit by $K_n$.
  \item We have $d(K_n,K_n^m)\le2w_m$ whenever $n\in[n_m]$.
  \item The sequence $\ve{K}=(K_n)_{n\in\N}$ belongs to $\Ar$.
 \end{enumerate}
\end{proposition}

\begin{proof}
 Remember the following:
 \begin{itemize}
  \item if $n\in[n_m]$, then $d(K_n^{m+1},K_n^m)<w_m$
        because $d(\tilde{K}_n^{m+1},\tilde{K}_n^m)<w_m$;
  \item $w_{m+1}<w_m/2$ and $\bar{U}(\ve{K}^m,n_m,2w_m)\subset\U_m$.
 \end{itemize}

 \begin{enumerate}
  \item Fix $n\in\N$ and denote by $m_0$ the least positive integer with $n\in[n_{m_0}]$.
        Then, since $d(K_n^{m+1},K_n^m)<w_m$ for all $m\ge m_0$,
        we have, for all $m$ and $m'$ with $m_0\le m<m'$,
        \[
         d(K_n^{m'},K_n^m)
         \le\sum_{k=m}^{m'-1}d(K_n^{k+1},K_n^k)
         <\sum_{k=m}^{m'-1}w_k
         \le\sum_{k=m}^{m'-1}2^{-(k-m)}w_m<2w_m.
        \]
        It follows that $(K_n^m)_{m\in\N}$ is a Cauchy sequence and therefore converges.
  \item Obvious from the estimate in the proof of (1).
  \item It follows from (2) that
        \[
         \ve{K}\in\bigcap_{m=1}^{\infty}\bar{U}(\ve{K}^m,n_m,2w_m)
         \subset\bigcap_{m=1}^{\infty}\U_m\subset\Ar.\qedhere
        \]
 \end{enumerate}
\end{proof}

\begin{proposition}\label{prop:N(f)_approximation}
 If $f\in\bigcap_{m=1}^{\infty}B(g_m,\beta_m)$, then
 \[
  N(f,a_j)\subset\bigcup_{n\in A_j^m}B(K_n,3w_m)\quad\text{and}\quad
  \bigcup_{n\in A_j^m}K_n\subset B\bigl(N(f,b_j),3w_m\bigr)
 \]
 whenever $j\le m$.
\end{proposition}

\begin{proof}
 Suppose that $j\le m$. Then by the choice of $\beta_m$ (Claim~\ref{claim:beta_m}),
 we have
 \begin{itemize}
  \item $\tilde{N}(f,a_j^{m,4})\subset\bigcup_{n\in A_j^m}B(\tilde{K}_n^m,w_m)$;
  \item $\bigcup_{n\in A_j^m}\tilde{K}_n^m
         \subset B\bigl(\tilde{N}(f,b_j^{m,3}),w_m\bigr)$.
 \end{itemize}
 Taking the union for $\hat{\ }$ and $\check{\ }$ gives
 \[
  N(f,a_j^{m,4})\subset\bigcup_{n\in A_j^m}B(K_n^m,w_m)\quad\text{and}\quad
  \bigcup_{n\in A_j^m}K_n^m\subset B\bigl(N(f,b_j^{m,3}),w_m\bigr).
 \]
 Therefore Proposition~\ref{prop:K_n^m_convergence} (2) shows that
 \begin{gather*}
  N(f,a_j)\subset N(f,a_j^{m,4})
          \subset\bigcup_{n\in A_j^m}B(K_n^m,w_m)
          \subset\bigcup_{n\in A_j^m}B(K_n,3w_m),\\
  \begin{aligned}
   \bigcup_{n\in A_j^m}K_n
   &\subset\bigcup_{n\in A_j^m}\bar{B}(K_n^m,2w_m)
    \subset B\bigl(N(f,b_j^{m,3}),3w_m\bigr)\\
   &\subset B\bigl(N(f,b_j),3w_m\bigr).
  \end{aligned}\qedhere
 \end{gather*}
\end{proof}

\begin{proposition}\label{prop:(K,f)inX}
 If $f\in\bigcap_{m=1}^{\infty}B(g_m,\beta_m)$, then $(\ve{K},f)\in\X$.
\end{proposition}

\begin{proof}
 Remember that if $2\le j\le m-1$, then
 \[
  \bigcup_{n\in A_j^m\setminus A_j^{m-1}}\tilde{K}_n^m
  \subset\bigcup_{n\in A_{j-1}^{m-1}}B(\tilde{K}_n^m,2w_{m-1}),
 \]
 and so the same inclusion holds when $\tilde{K}_n^m$ is replaced by $K_n^m$:
 \[
  \bigcup_{n\in A_j^m\setminus A_j^{m-1}}K_n^m
  \subset\bigcup_{n\in A_{j-1}^{m-1}}B(K_n^m,2w_{m-1}).
 \]
 Therefore Proposition~\ref{prop:K_n^m_convergence} (2) shows that
 \begin{align*}
  \bigcup_{n\in A_j^m\setminus A_j^{m-1}}K_n
  &\subset\bigcup_{n\in A_j^m\setminus A_j^{m-1}}\bar{B}(K_n^m,2w_m)
   \subset\bigcup_{n\in A_{j-1}^{m-1}}B(K_n^m,2w_m+2w_{m-1})\\
  &\subset\bigcup_{n\in A_{j-1}^{m-1}}B(K_n,4w_m+2w_{m-1})
 \end{align*}
 whenever $2\le j\le m-1$.
 Hence if we define $\ve{\delta}\in Y$ by $\delta_m=4w_m+2w_{m-1}$ for $m\in\N$,
 then, using Proposition~\ref{prop:N(f)_approximation}, we may conclude that
 \begin{itemize}
  \item $\bigcup_{n\in A_j^m\setminus A_j^{m-1}}K_n
         \subset\bigcup_{n\in A_{j-1}^{m-1}}B(K_n,\delta_m)$
        whenever $2\le j\le m-1$, i.e.~%
        $\ve{K}\in\Sr_0(\ve{n},\ve{\delta})$;
  \item $N(f,a_j)\subset\bigcup_{n\in A_j^m}B(K_n,\delta_m)$ whenever $j\le m$;
  \item $\bigcup_{n\in A_j^m}K_n\subset B\bigl(N(f,b_j),\delta_m\bigr)$
        whenever $j\le m$.
 \end{itemize}
 It follows that
 $(\ve{K},f,\ve{n},\ve{\delta},\ve{a},\ve{b})\in\Y_0$, implying that $(\ve{K},f)\in\X$.
\end{proof}

\begin{proposition}
 We have $\bigcap_{m=1}^{\infty}B(g_m,\beta_m)\subset S$.
 Hence the strategy makes Player~II win.
\end{proposition}

\begin{proof}
 Immediate from Proposition~\ref{prop:K_n^m_convergence} (3)
 and Proposition~\ref{prop:(K,f)inX}.
\end{proof}

This completes the proof of the key proposition (Proposition~\ref{prop:key_prop})
and hence the main theorem has been proved.

\bibliographystyle{amsplain}

\end{document}